\documentclass[reqno]{amsart}
\usepackage{amsmath, amsthm, amssymb, amsfonts}
\usepackage{mathtools}
\usepackage{subcaption}
\usepackage[normalem]{ulem}
\usepackage{graphicx}
\usepackage{booktabs}
\usepackage{xcolor}
\usepackage[foot]{amsaddr}
\usepackage{hyperref}
\hypersetup{
    colorlinks=true,
    citecolor=green!50!black,
    linkcolor=red!80!black,
    filecolor=blue,
    urlcolor=blue!70!black,
}
\newcommand{\stkout}[1]{\ifmmode\text{\sout{\ensuremath{#1}}}\else\sout{#1}\fi}

\newtheorem{theorem}{Theorem}
\newtheorem{proposition}[theorem]{Proposition}
\newtheorem{corollary}[theorem]{Corollary}
\newtheorem{lemma}[theorem]{Lemma}
\theoremstyle{definition}
\newtheorem{remark}{Remark}
\newtheorem{definition}{Definition}

\newcommand{\arxivonly}[1]{#1}
\newcommand{\journalonly}[1]{}

\newcommand{\new}[1]{#1}
\newcommand{\old}[1]{}

\title{Two-dimensional fluids via matrix hydrodynamics}

\author{Klas Modin${^{1\journalonly{,*}}}$}
\address{${^1}$Department of Mathematical Sciences, Chalmers University of Technology and University of Gothenburg, SE-412~96 Gothenburg, Sweden}
\arxivonly{
\email{klas.modin@chalmers.se}
}
\journalonly{
\thanks{${}^*$Corresponding author, \texttt{klas.modin@chalmers.se}, \href{https://orcid.org/0000-0001-6900-1122}{ORCID:0000-0001-6900-1122}}
}

\author{Milo Viviani${^{2\journalonly{,\dagger}}}$}
\address{${^2}$Scuola Normale Superiore di Pisa, Pisa, Italy}
\arxivonly{
\email{milo.viviani@sns.it}
}
\journalonly{\thanks{${}^\dagger$\texttt{milo.viviani@sns.it}, \href{https://orcid.org/0000-0002-2340-0483}{ORCID:0000-0002-2340-0483}}}


\date{\today}

\newcommand{\stab}{\mathrm{stab}}
\newcommand{\su}{\mathfrak{su}}

\newcommand{\Ss}{\mathbb{S}}
\newcommand{\Tt}{\mathbb{T}}
\newcommand{\Rr}{\mathbb{R}}
\newcommand{\Zz}{\mathbb{Z}}
\newcommand{\Cc}{\mathbb{C}}

\newcommand{\bracketLP}[1]{\prec\! #1 \! \succ\!}

\begin{document}

\dedicatory{To the memory of Vladimir Zeitlin, the father of matrix hydrodynamics.}

\begin{abstract}
Two-dimensional (2-D) incompressible, inviscid fluids produce fascinating patterns of swirling motion. 
How and why the patterns emerge are long-standing questions, first addressed in the 19th century by Helmholtz, Kirchhoff, and Kelvin. 
Countless researchers have since contributed to innovative techniques and results. Yet, the overarching problem of swirling 2-D motion and its long-time behavior remains largely open. 
Here we shed light on this problem via a link to isospectral matrix flows. 
The link is established through V.\;Zeitlin's beautiful model for the numerical discretization of Euler's equations in 2-D. 
When considered on the sphere, Zeitlin's model offers deep connections between 2-D hydrodynamics and unitary representations of the rotation group. 
Consequently, it provides a dictionary that maps hydrodynamical concepts to matrix Lie theory, which in turn gives connections to matrix factorizations, random matrices, and integrability theory, for example. 
Results about finite-dimensional matrices can then be transferred to infinite-dimensional fluids via quantization theory, which is here used as an analysis tool (albeit traditionally describing the limit between quantum and classical physics).
We demonstrate how the dictionary is constructed and how it unveils techniques for 2-D hydrodynamics. 
We also give accompanying convergence results for Zeitlin's model on the sphere.
\end{abstract}

\maketitle

\vspace{-3ex}

\arxivonly{
\vspace{-5ex}
\tableofcontents
}

\journalonly{
\subsection*{Statements and Declarations}
The authors have no relevant financial or non-financial interests to disclose.
Data sharing is not applicable to this article as no datasets were generated or analyzed during the current study.

\subsection*{Acknowledgements}
This work was supported by the Swedish Research Council (grant number 2022-03453), the Knut and Alice Wallenberg Foundation (grant numbers WAF2019.0201 and KAW2020.0287), and the Göran Gustafsson Foundation for Research in Natural Sciences and Medicine.
The computations were enabled by resources provided by Chalmers e-Commons at Chalmers.
The data handling was enabled by resources provided by the National Academic Infrastructure for Supercomputing in Sweden (NAISS), partially funded by the Swedish Research Council through grant agreement no.~2022-06725.
We would like to thank Michele Dolce and Theo Drivas for fruitful discussions related to this work, and Michael Roop and Ali Suri for important comments related to the convergence proof in Appendix~\ref{appendix:convergence_proof}.
Furthermore, we thank the anonymous reviewers for many helpful suggestions.
}

\section{Introduction}

If you stir up an incompressible, low-viscosity fluid confined to a thin (almost 2-D) bounded domain, then, under conditions on the initial configuration and the applied forcing, what you see is spectacular, completely different from the situation in full 3-D.
The fluid self-organizes into large, coherently swirling regions called vortex condensates.\footnote{The effect can be experimentally reproduced by a thin soap film flowing rapidly through a fine comb~\cite{Co1984}, or by a conducting fluid confined to a thin layer and driven into turbulence by a temporally varying magnetic field~\cite{So1988}.}
Some regions swirl counterclockwise (positive vorticity), others clockwise (negative vorticity), with occasional merger between equal-signed regions, and repulsion between opposite-signed regions. 
What is the mechanism of vortex condensation and under which conditions does it occur?
More mathematically formulated, what is the generic long-time behavior of the 2-D Euler equations?

Towards an answer, a milestone was reached in 1949 when Onsager~\cite{On1949} applied statistical mechanics to a large but finite number of point vortices on a flat torus (doubly periodic square). 
Onsager understood that if the energy of a configuration is large, relative to the vortex strengths, then the thermodynamical temperature is negative. 
Point vortices of equal sign will then tend to cluster ``so as to use up excess energy at the least possible cost in terms of degrees of freedom'' \cite[p.~281]{On1949}.
Statistical mechanics thus predicts vortex condensation.
Under the ergodicity assumption, it also predicts long-time states corresponding to thermal equilibrium.
Since then, Onsager's ideas have fostered several rigorous results, for example, weak convergence of the Gibbs measure, in both positive and negative temperature regimes, as the number of point vortices increase \cite{CaLiMaPu1992,Ki1993}.

But there is a problem with the point vortex approach, as Onsager himself pointed out: ``When we compare our idealized model with reality, we have to admit one profound difference: the distribution of vorticity which occur in the actual flow of normal liquids are continuous'' \cite[p.~281]{On1949}.
In particular, solutions to the 2-D Euler equations with continuous, bounded vorticity possess a rich geometric structure: they make up an infinite-dimensional Lie--Poisson system, with infinitely many conservation laws given by Casimir functions (\emph{cf.}~Sec.~\ref{sec:background} below).
Whereas point-vortices formally fit into this structure, with Casimirs corresponding to the vortex strengths, they do it in a weak sense, which fails to capture the incompressible nature of the fluid, for example.
In addition, point vortices cannot capture relations between the continuous Casimirs, and those relations affect the long-time behavior (\emph{cf.}~Abramov and Majda~\cite{AbMa2003b}).
Moreover, the ergodicity assumption for Onsager's statistical approach is invalid.\footnote{This does not, however, wholly discredit the statistical theory; see the discussion in Marchioro and Pulvirenti~\cite[Sec.~7.5]{MaPu2012}.}
So, in summary, the point vortex model is insufficient to describe the long-time behavior of continuous solutions.

Another established technique for 2-D hydrodynamics is non-linear PDE analysis.
It has led to many rigorous and deep results, but at the cost of restricted settings, for example perturbations of steady solutions.
We refer to the monograph of Marchioro and Pulvirenti~\cite{MaPu2012} and the lecture notes of {\v{S}}verák~\cite{Sv2011} for an overview.

The grand vision is to bridge PDE analysis results with predictions from statistical hydrodynamics.
It is a notoriously difficult problem.
As a flavor of the difficulty, the two approaches seem incompatible:
The statistically predicted vortex condensation implies merging of equally signed vorticity regions.
Yet, in the dynamics of the \mbox{2-D} Euler equations, vorticity is advected by the fluid velocity field, so merging can only occur by trapping increasingly narrow, but deep, vorticity variations in an overwhelmingly complicated stretch-and-fold process.
These variations cannot disappear in the $C^1$-topology.
Indeed, the $C^1$-norm along regular solutions can grow extremely fast, up to double exponential \cite{KiSv2014}, and in numerical simulations it typically does grow extremely fast.

\begin{figure*}
    \centering
    \begin{subfigure}{0.325\textwidth}
        \includegraphics[width=\textwidth]{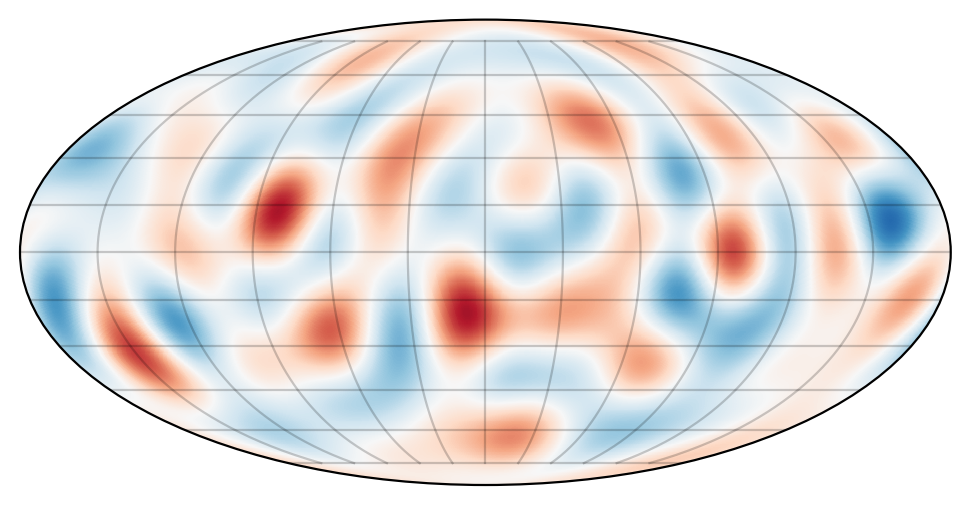}
        \caption{Initial vorticity $\omega_0$}
        \label{fig:non-vanish-initial}
    \end{subfigure}
    \begin{subfigure}{0.325\textwidth}
        \includegraphics[width=\textwidth]{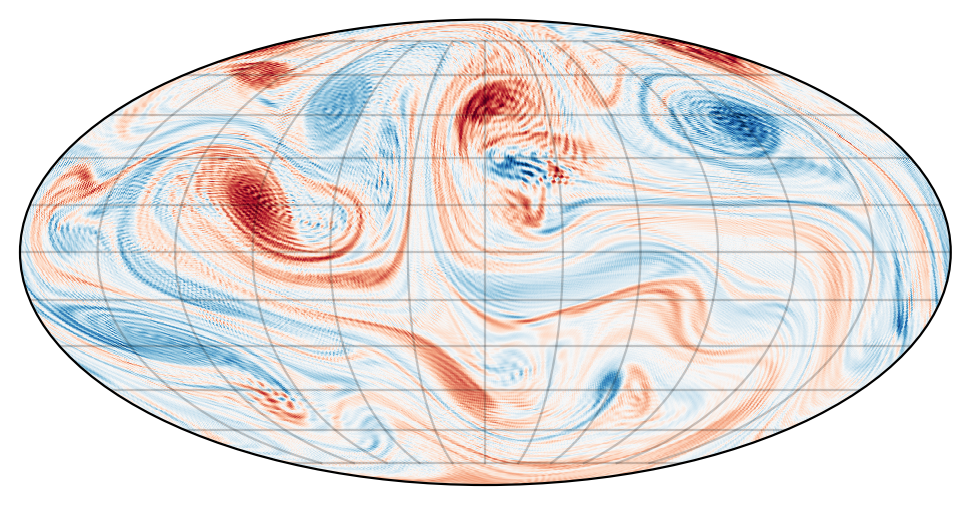}
        \caption{Mixing phase}
        \label{fig:non-vanish-mixing}
    \end{subfigure}
    \begin{subfigure}{0.325\textwidth}
        \includegraphics[width=\textwidth]{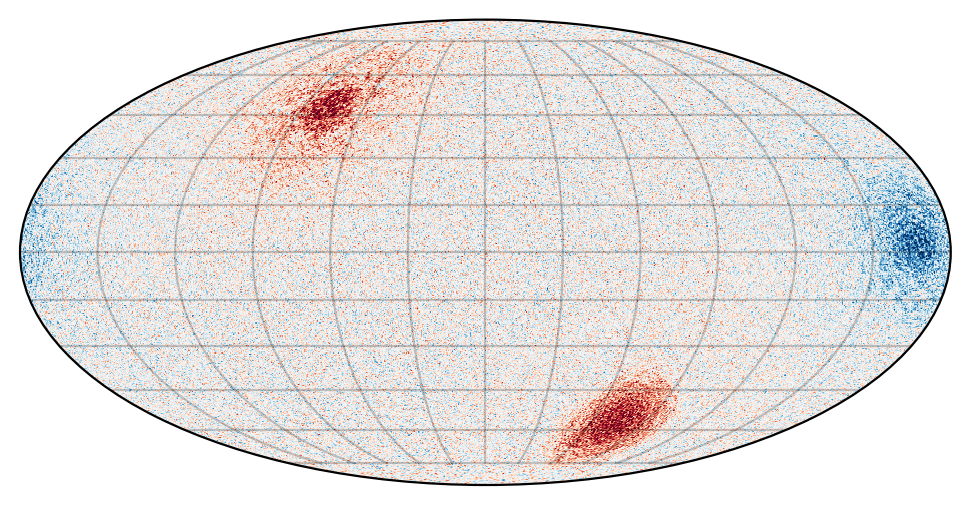}
        \caption{Long-time phase}
        \label{fig:non-vanish-final}
    \end{subfigure}
    \caption{
        A numerical simulation of the vorticity field for Euler's equations on $\Ss^2$. 
        The results are displayed using the Mollweide area preserving projection.
        For random smooth initial data \textsc{(a)}, the typical dynamical behavior is a mixing phase \textsc{(b)} where vorticity regions of equal sign undergo merging, followed by a long-time phase \textsc{(c)} of 3 or 4 remaining large, weakly interacting vortex condensates whose centers of mass move along nearly quasi periodic trajectories.
    }
    \label{fig:non-vanishing}
\end{figure*}

A third tool for 2-D hydrodynamics is to carry out numerical experiments (as illustrated in Fig.~\ref{fig:non-vanishing}, where the mixing phase \textsc{(b)} captures a typical stretch-and-fold process).
Indeed, since the work of Lorentz~\cite{Lo1962} in the 1960s, computer simulations have guided many analytical results in fluid dynamics; recent examples are found in the references \cite{Gr1990, Sh2013, KiSv2014, ElJe2020, ChHoHu2021,WaLaGoBu2022}.
But are the numerical outcomes consistent with the geometry known to influence the long-time behavior?

To address this question, we can ask for the finite-dimensional dynamical model resulting from the discretization to be itself a Lie--Poisson system, whose finite-dimensional Lie--Poisson bracket approximates the infinite-dimen\-sional one.
Such a model was given by Zeitlin~\cite{Ze1991,Ze2004}, based on quantization results of Hoppe~\cite{Ho1982,Ho1989}. 
It is the only known Lie--Poisson discretization of the 2-D Euler equations, and numerical simulations indicate that it captures the qualitative behavior better than traditional discretizations (e.g., the predicted spectral scaling law, which standard methods fail to catch even at significantly higher resolutions~\cite{CiViLuMoGe2022}).
Yet, the model is vastly underexplored.
%


This paper began with a simple idea: in addition to numerical advantages, \emph{matrix hydrodynamics} -- the matrix flows enabled via Zeitlin's model -- brings theoretical insights to classical 2-D hydrodynamics by providing a link to the rich and well-developed theory of matrices. 
The point is to put forward matrix hydrodynamics as a tool for 2-D turbulence, similar in spirit to the point vortex model of Onsager, but with the benefit of capturing the Lie--Poisson geometry of continuous vorticity fields, including conservation of continuous Casimirs.

Besides the examples we give throughout the paper, a motivation and feasibility check for our idea comes from complex geometry.
Indeed, in the late 1980s, quantization was promoted to address questions about Kähler--Einstein metrics and the Calabi conjecture.
Tian~\cite{Ti1990} then gave convergence results, since used regularly in theoretical contexts.
Only much later did Donaldson~\cite{Do2009} use quantization in a numerical study of Kähler--Einstein metrics; the antithesis to the developments in 2-D hydrodynamics, where Zeitlin's model has been a numerical tool, until recently.

Our primary aim with the paper is to give basic convergence results for matrix hydrodynamics in the spherical setting (Theorem~\ref{thm:convergence}, Theorem~\ref{thm:spectral_measure_convergence}, Corollary~\ref{thm:consistency}, and Theorem~\ref{thm:space-time-convergence-detailed}), and to showcase the fluid-to-matrix link in selected examples and relate it to existing theories (Sec.~\ref{sec:coadjoint-orbits}-\ref{sec:integrability-and-long-time}).
Our secondary aim, addressed in Sec.~\ref{sec:quantization-and-EZ}, is to give a transparent derivation of Zeitlin's model on the sphere, based on unitary representation theory and focused on geometry rather than algebraic formulae.
But first, we review in Sec.~\ref{sec:background} the geometry of the 2-D Euler equations.

\arxivonly{
\medskip
\noindent\textbf{Acknowledgements} 
This work was supported by the Swedish Research Council (grant number 2022-03453), the Knut and Alice Wallenberg Foundation (grant numbers WAF2019.0201 and KAW2020.0287), and the Göran Gustafsson Foundation for Research in Natural Sciences and Medicine.
The computations were enabled by resources provided by Chalmers e-Commons at Chalmers.
The data handling was enabled by resources provided by the National Academic Infrastructure for Supercomputing in Sweden (NAISS), partially funded by the Swedish Research Council through grant agreement no.~2022-06725.
We would like to thank Michele Dolce and Theo Drivas for fruitful discussions related to this work, and Michael Roop and Ali Suri for important comments related to the convergence proof in Appendix~\ref{appendix:convergence_proof}.
}

\section{Background: geometry of the 2-D Euler equations}\label{sec:background}

For an incompressible, inviscid fluid, Euler~\cite{Eu1757} showed how Newton's second law leads to partial differential equations (PDEs) for the velocity field $\boldsymbol{v}(x,t)$ of the fluid's motion,
\begin{equation}\label{eq:euler}
       \frac{\partial \boldsymbol{v}}{\partial t} + \nabla_{\boldsymbol{v}}\boldsymbol{v} = -\nabla p ,\quad \operatorname{div}\boldsymbol{v} = 0,
\end{equation}
where $p$ is the pressure function and $\nabla_{\boldsymbol{v}}$ is the co-variant derivative (furthermore, $\boldsymbol{v}$ is tangential to the boundary if present).
Euler's equations~\eqref{eq:euler} make sense on any Riemannian manifold, but in 2-D the geometric structure is richer than in higher dimensions, and the dynamics is significantly different.
Throughout this paper we shall work on the sphere $\Ss^2$, partly because of its relevance in geophysical contexts, but mostly because it later allows a description in terms of unitary representation theory that renders the geometric structures in Zeitlin's model more transparent.

Let $J\colon T \Ss^2 \to T \Ss^2$ denote the bundle mapping for rotation by $\pi/2$ in positive direction.
The curl operator maps the vector field $\boldsymbol{v}$ on $\Ss^2$ to the \emph{vorticity function} 
\begin{equation*}
    \omega = \operatorname{curl}\boldsymbol{v} \equiv \operatorname{div}(J\boldsymbol{v}).    
\end{equation*}
Since $\operatorname{curl}\circ \nabla \equiv 0$, we obtain from Euler's equations \eqref{eq:euler} that $\omega$ fulfills the transport equation 
\begin{equation*}
    \frac{\partial \omega}{\partial t} + \operatorname{div}(\omega \boldsymbol{v}) = 0.
\end{equation*}
Further, since the first co-homology of the sphere is trivial, it follows from Hodge theory that the divergence free vector field $\boldsymbol{v}$ can be written $\boldsymbol{v} = \nabla^\bot \psi \equiv -J\nabla\psi$, where $\psi$ is the \emph{stream function}, unique up to a constant.
Euler's equations \eqref{eq:euler} is then formulated entirely in terms of the vorticity and stream functions
\begin{equation}\label{eq:vorteq}
    \frac{\partial \omega}{\partial t} + \operatorname{div}(\omega \nabla^\bot \psi) = 0, \quad -\Delta\psi = \omega.
\end{equation}

Following Arnold~\cite{Ar1966} and Marsden and Weinstein~\cite{MaWe1983},
we now demonstrate how to arrive at these equations from the point-of-view of symplectic geometry and Hamiltonian dynamics.
 
The sphere is a symplectic manifold with symplectic form given by the spherical area form $\mu$ (in spherical coordinates $(\theta,\phi) \in [0,\pi]\times [0,2\pi)$ the area form is $\mu=\sin\theta\,\mathrm{d}\theta\wedge \mathrm{d}\phi$).
Given a Hamiltonian function $f$ on $\Ss^2$, the corresponding symplectic vector field $X_f$ is defined by 
\begin{equation}\label{eq:hameq}
     \iota_{X_f}\mu = \mathrm{d}f. 
\end{equation}
If we identify 1-forms with vector fields using the Riemannian structure of $\Ss^2$, equation~\eqref{eq:hameq} becomes 
\begin{equation*}
   J X_f = \nabla f \iff X_f = -J \nabla f.
\end{equation*}
Thus, the fluid velocity field is the symplectic vector field for the Hamiltonian given by the stream function: $\boldsymbol{v} = X_\psi$.
 
The space $\mathfrak{X}_\mu(\Ss^2)$ of smooth, symplectic vector fields on $\Ss^2$ forms an infinite-dimensional Lie algebra (with the vector field bracket $[\cdot,\cdot]_\mathfrak{X}$).
This algebra is isomorphic to the Poisson algebra of smooth functions modulo constants $C^\infty(\Ss^2)/\Rr$ via the mapping $\psi \mapsto X_\psi$.
Indeed,
\begin{equation*}
    -[X_\psi,X_\xi]_\mathfrak{X} = X_{\{\psi,\xi \}} \qquad\text{where}\qquad \{\psi,\xi \} = \nabla\psi\cdot \nabla^\bot \xi .
\end{equation*}
Let us now describe the connection between the infinite-dimensional Lie algebra $\mathfrak{X}_\mu(\Ss^2)$ and the vorticity equation~\eqref{eq:vorteq}.

Let $\mathfrak{g}$ be a Lie algebra. 
Its dual $\mathfrak{g}^*$ is a Poisson manifold via its \emph{Lie--Poisson bracket} defined on functions $F,H \in C^\infty(\mathfrak{g}^*)$ by
\begin{equation}\label{eq:LP_bracket_abstract}
    \bracketLP{F, H}(\omega) = \langle\omega, [\mathrm{d}F(\omega), \mathrm{d}H(\omega)] \rangle , \qquad \omega\in\mathfrak{g}^*,
\end{equation}
where $[\cdot,\cdot]$ denotes the Lie bracket on $\mathfrak{g}$ and $\langle\cdot,\cdot\rangle$ is the pairing between $\mathfrak{g}^*$ and $\mathfrak{g}$.
The Lie-Poisson bracket~\eqref{eq:LP_bracket_abstract} originates from the canonical symplectic structure on the co-tangent bundle of the corresponding Lie group $G$, via symmetry reduction $T^*G/G\simeq \mathfrak{g}^*$. 
For details on this framework and its application in hydrodynamics, we refer to Arnold and Khesin~\cite{ArKh1998}.

In the infinite-dimensional case of the group of symplectomorphisms on $\Ss^2$ and its Lie algebra of smooth Hamiltonian (or stream) functions, i.e., $G=\mathrm{Diff}_\mu(\Ss^2)$ and $\mathfrak{g} = C^\infty(\Ss^2)/\Rr$, one usually restricts to the \emph{smooth dual}, which is constructed so that $C^\infty(\Ss^2)^* \simeq C^\infty(\Ss^2)$ via the $L^2$-pairing $\langle \omega,\psi \rangle = \int_{\Ss^2} \omega\psi\mu$.
As indicated, we think of vorticity $\omega$ as a dual variable to the stream function $\psi\in \mathfrak{g} = C^\infty(\Ss^2)/\Rr$.
Thus, $$\omega \in \mathfrak{g}^* = (C^\infty(\Ss^2)/\Rr)^* \simeq C_0^\infty(\Ss^2) = \{f\in C^\infty(\Ss^2)\mid \int_{\Ss^2} f \mu = 0 \}. $$
So the smooth dual of the Lie algebra of Hamiltonian functions (or equivalently symplectic vector fields) is given by smooth functions with vanishing mean.

Now, the Hamiltonian system on the Lie--Poisson space $\mathfrak{g}^*$ for a Hamiltonian function $H\colon \mathfrak{g}^* \to \Rr$ is formally given by
\begin{equation}\label{eq:LiePoisson}
    \dot \omega + \operatorname{ad}^*_{\frac{\delta H}{\delta \omega}} \omega = 0
\end{equation}
where $\dot\omega$ denotes time differentiation and $\operatorname{ad}^*_\psi\colon \mathfrak{g}^*\to\mathfrak{g}^*$ is defined by 
$
    \langle \operatorname{ad}^*_\psi \omega,\xi  \rangle = \langle \omega, \lbrace \psi, \xi \rbrace \rangle
$ for all $\xi\in\mathfrak{g}$.
In the case $\mathfrak{g} = C^\infty(\Ss^2)/\Rr$, with smooth dual $\mathfrak{g}^* = C^\infty_0(\Ss^2)$, we obtain from the divergence theorem that
\begin{equation*}
    \langle \omega, \lbrace \psi, \xi \rbrace \rangle  = -\int_{\Ss^2} \omega \nabla^\bot \psi\cdot \nabla \xi\, \mu 
    = \langle \operatorname{div}(\omega \nabla^\bot \psi),\xi \rangle
    = \langle \underbrace{\{ \omega, \psi\} }_{\operatorname{ad}^*_\psi\omega},\xi \rangle.
\end{equation*}
Hence, $\operatorname{ad}^*_\psi\omega$ is minus the Poisson bracket, which reflects that the $L^2$-pairing on the Poisson space $C^\infty(\Ss^2)$ is bi-invariant. 

Equation \eqref{eq:vorteq}, for the evolution of vorticity, can now be written as a Lie--Poisson system for the quadratic Hamiltonian
\begin{equation*}
    H(\omega) = \frac{1}{2}\int_{\Ss^2} \left| \nabla \psi \right|^2 \, \mu = \frac{1}{2}\int_{\Ss^2} \omega\underbrace{(-\Delta)^{-1}\omega}_{\psi} \, \mu    .
\end{equation*}
Its variational derivative is $\frac{\delta H}{\delta \omega} = \psi$.
Consequently, the Lie--Poisson form~\eqref{eq:LiePoisson} of equation \eqref{eq:vorteq} is
\begin{equation}\label{eq:vorteq2}
    \dot \omega + \lbrace \omega,\psi\rbrace = 0, \quad -\Delta\psi = \omega, 
\end{equation}
where we think of the Laplacian as an isomorphism between $\mathfrak{g}$ and $\mathfrak{g}^*$, i.e., $$\Delta\colon C^\infty(\Ss^2)/\Rr \to C^\infty_0(\Ss^2).$$
Equation \eqref{eq:vorteq2} is the starting point for Zeitlin's model: the notion is to replace the infinite-dimensional Lie algebra $\mathfrak{g}$ with a finite-dimensional one and then consider the corresponding Lie--Poisson system.
We give the details in the next section.
Here, we continue with a re-formulation of equation \eqref{eq:vorteq2} that reveals the connection to the Lie group of $\mathfrak{g}$.

From \eqref{eq:vorteq2} it is clear that the vorticity function $\omega$ is infinitesimally transported by the time-dependent Hamiltonian vector field $\boldsymbol{v} = X_\psi$.
Consequently, the flow map $\Phi\colon \Ss^2\times\Rr\to \Ss^2$ obtained by integrating the vector field 
\begin{equation*}
    \frac{\partial {\Phi(x,t)}}{\partial t} = \boldsymbol{v}\big(\Phi(x,t), t\big),
    \quad \Phi(x,0) = x,
\end{equation*}
is a curve $t\mapsto \Phi(\cdot,t)$ in the Lie group $G$ corresponding to the Lie algebra $\mathfrak{g}$.
Since the integration of a vector field yields a symplectic map if and only if the vector field is symplectic, it follows that the Lie group consists of symplectic diffeomorphisms of $\Ss^2$, denoted $\operatorname{Diff}_\mu(\Ss^2)$.
Since $\omega$ is transported, we have that $\omega(x,t) = \omega(\Phi^{-1}(x,t),0)$, or in more compact notation $\omega_t = \omega_0\circ\Phi_t^{-1}$.
The direct implication is that $\omega$ remains on the same \emph{coadjoint orbit} as $\omega_0$ (the orbits for the action of $\operatorname{Diff}_\mu(\Ss^2)$ on $C^\infty_0(\Ss^2)$).
Therefore, since the map $\Phi_t$ has unitary Jacobian determinant, any functional of the form
\begin{equation*}
    C_f(\omega) = \int_{\Ss^2} f\big(\omega(x,t)\big)\,\mu(x),\quad f\colon\Rr\to\Rr,
\end{equation*}
is conserved: these are the \emph{Casimir functions}.

The transport property furthermore implies weak (distributional), finite-dimensional coadjoint orbits given by $n$ interacting point vortices 
\begin{equation}\label{eq:PV}
    \omega = \sum_{k=1}^n \Gamma_k \delta_{x_k}.    
\end{equation}
These were found by Helmholtz~\cite{He1858}, and analyzed by Kirchhoff~\cite{Ki1876} and Kelvin~\cite{Ke1887}.
Much later, Onsager~\cite{On1949} used them is his theory of statistical hydrodynamics.
Let us see how they fit into the geometric framework of Lie-Poisson dynamics, following Marsden and Weinstein~\cite{MaWe1983}.
In general, if $G$ acts on a symplectic manifold $(M,\nu)$ by symplectomorphisms then there is a corresponding infinitesimal action vector field $\xi_M \in \mathfrak{X}_\nu(M)$ of $\xi\in\mathfrak{g}$.
Suppose that $\xi_M$ is Hamiltonian, i.e., there is a function $H_\xi\in C_0^\infty(M)$ such that $\iota_{\xi_M}\nu = \mathrm{d} H_\xi$.
Notice that $H_\xi$ must be linear in $\xi$.
Suppose moreover there is a smooth function $p\colon M\to \mathfrak{g}^*$ such that 
\begin{equation*}
    H_\xi(m) = \langle p(m),\xi \rangle .
\end{equation*}
Then $p$ is a \emph{momentum map}.
Often (but not always) we can take the momentum map to be equivariant, i.e., $p(g\cdot m) = \mathrm{Ad}^*_g(p(m))$.
If, in addition, the action of $G$ on $M$ is transitive, then the image $p(M)\subset \mathfrak{g}^*$ is a coadjoint orbit (generally weak).

In our specific situation, $G=\mathrm{Diff}_\mu(\Ss^2)$, we have thus far seen the tautological example, where $M$ is a symplectic leaf $\mathcal{O}_{\omega_0}$ of the Poisson structure restricted to the smooth dual $\mathfrak{g}^* = C_0^\infty(\Ss^2)$, and the momentum map is the inclusion $p\colon \mathcal{O}_{\omega_0}\to \mathfrak{g}^*$.
But $\mathrm{Diff}_\mu(\Ss^2)$ also acts symplectically on $\Ss^2$, by definition.
Via Cartesian extension, there is then a symplectic action of $\mathrm{Diff}_\mu(\Ss^2)$ on $M = (\Ss^2)^n\backslash \{ \text{diagonal} \}$ equipped with the symplectic structure
\begin{equation*}
    \nu_{(x_1,\ldots,x_n)} = \sum_{k=1}^n\Gamma_k \mu_{x_k},\qquad \Gamma_k \in \mathbb{R}\backslash \{0\}.
\end{equation*}
Since the action is transitive, we expect corresponding finite dimensional coadjoint orbits.
Indeed, $\xi_M(x_1,\ldots,x_n) = \big(\xi(x_1),\ldots,\xi(x_n)\big)$ so
\begin{equation*}
    \iota_{\xi_M}\nu = \sum_k \Gamma_k \iota_{\xi(x_k)}\mu = \sum_k \Gamma_k \mathrm{d} \psi(x_k) = \mathrm{d} \sum_{k}\Gamma_k \psi(x_k) = \mathrm{d}  \int_{\Ss^2}\underbrace{\left(\sum_{k} \Gamma_k\delta_{x_k}\right)}_{\omega = p(x_1,\ldots,x_n)} \psi \mu   ,
\end{equation*}
where $\xi = X_\psi$.
Thus, the associated finite dimensional (weak) coadjoint orbit is given by the point vortices \eqref{eq:PV}.
These orbits and the corresponding weak solutions show that the restriction to the smooth dual is not always legitimate.

To make the analogy with continuous vorticity functions, one should think of phase space for point vortices as parameterized by $(x_1,\ldots,x_n, \Gamma_1,\ldots,\Gamma_n)$.
The positions $x_k$ correspond to the level sets of a smooth vorticity function, and the strengths $\Gamma_k$ to the values associated with the level sets.
Just like the function values of the initial conditions $\omega_0\in C^\infty(\Ss^2)$ are preserved (since $\omega_t$ is transported), the strengths $\Gamma_k$ are conserved.
The dynamics for point vortices take place in the map $(x_1(0),\ldots,x_n(0))\mapsto (x_1(t),\ldots,x_n(t))$ which moves the strengths around.
The coordinates $(x_1,\ldots,x_n, \Gamma_1,\ldots,\Gamma_n)$ for point vortices have a direct analog for smooth coadjoint orbits.
%
%
%
%
Indeed, instead of $\omega_t$ as dynamic variable, we can use the map $\Phi_t \in \operatorname{Diff}_\mu(\Ss^2)$ together with the initial conditions $\omega_0$ via the equation
\begin{equation}\label{eq:transport_form}
    \dot\Phi_t = X_\psi\circ\Phi_t, \quad -\Delta\psi = \omega_0\circ\Phi_t^{-1}.
\end{equation}
%
%
This is the \emph{transport map formulation} of equation \eqref{eq:vorteq2}.
It is the basis for global well-posedness results (\emph{cf.}~\cite{Yu1963,MaPu2012}), which shows convergence of the fixed point iteration obtained by leap-frogging between the equations~\eqref{eq:transport_form} for $\Phi$ (Picard iterations for ODE) and for $\psi$ (stationary PDE).
Geometrically, the transport formulation captures that the symplectic leaves are generated by the action of the Lie group.
 
We have thus summarized the geometric description of the 2-D Euler equations on $\Ss^2$.
Indeed, the structures presented -- the Hamiltonian formulation, the coadjoint orbits, the conservation of Casimirs, and the transport formulation -- are all important in the various mathematical theories  for long-time behavior (\emph{cf.}~surveys by Drivas and Elgindi~\cite{DrEl2023} and Khesin, Misio{\l }ek, and Shnirelman~\cite{KhMiSh2022}).

\begin{remark}
    Arnold's~\cite{Ar1966} geometric description is valid for Euler's equations on Riemannian manifolds $M$ of arbitrary dimension, by taking the group $G$ to be volume preserving diffeomorphisms $\mathrm{Diff}_{\mathrm{vol}}(M)$ with the Lie algebra $\mathfrak{g}$ of divergence free vector fields.
    However, this setting fails to capture some geometric qualities of 2-D Euler.
    In particular, the action of $\mathrm{Diff}_{\mathrm{vol}}(M)$ on $M$ is not symplectic if $\operatorname{dim} M>2$, nor is the action on $C^\infty(M)$ a Poisson map, which, for example, imply that singular and smooth coadjoint orbits cannot be given as point vortices and vorticity functions.
    Instead, the geometrically natural generalization of 2-D Euler is the group of (Hamiltonian) symplectomorphisms of a symplectic manifold equipped with a Riemannian metric (e.g., a Kähler manifold).
    This setting captures all the geometry of the 2-D Euler equations and therefore many of the analysis results (e.g.,~global solutions~\cite{Eb2012}).
\end{remark}

\section{Quantization and the Euler--Zeitlin equation}\label{sec:quantization-and-EZ}

Here are the steps leading to Zeitlin's model on the sphere:
\begin{enumerate}
    \item The first ingredient is an explicit, finite-dimensional quantization of the Poisson algebra $(C^\infty(\Ss^2),\{\cdot,\cdot\})$. 
    That is, a sequence of linear maps $$T_N\colon C^\infty(\Ss^2)\to \mathfrak{u}(N)$$ such that the Poisson bracket is approximated by the matrix commutator in the spectral norm $\lVert\cdot\rVert_\infty$:
    \[
        \lVert T_N(\{\psi,\omega \}) - \frac{1}{\hbar} [T_N\psi,T_N\omega]\rVert_\infty \to 0
        \quad\text{as}\quad N\to \infty,
    \]
    where $\hbar \sim 1/N$.
    Notice that the convergence cannot be uniform in $\psi$ and $\omega$, since $C^\infty(\Ss^2)$ is infinite-dimensional.

    \item The projection maps $T_N$ should take the unit function $x\mapsto 1$ to the imaginary identity matrix $iI\in \mathfrak{u}(N)$.
    In particular, $T_N$ descends to a map between the quotient Lie algebras $C^\infty(\Ss^2)/\mathbb{R}$ and $\mathfrak{pu}(N) \equiv \mathfrak{u}(N)/i\mathbb{R}I$.\footnote{We refrain from identifying the Lie algebra $\mathfrak{pu}(N)$ with $\mathfrak{su}(N)$, since its Lie group $\mathrm{PU}(N)$ is different from $\mathrm{SU}(N)$.}

    \item The dual $\mathfrak{pu}(N)^*$ is naturally identified with $\mathfrak{su}(N)$ via the pairing $\langle W,P\rangle = -\operatorname{tr}(WP)$ for $P\in\mathfrak{pu}(N)$ and $W\in\mathfrak{su}(N)$.
    The corresponding $\operatorname{ad}^*$-operator on $\mathfrak{su}(N)$ is
    \[
        \operatorname{ad}^*_PW = \frac{1}{\hbar}[W,P].  
    \]

    \item The last ingredient is a ``quantized'' Laplace operator $\Delta_N\colon\mathfrak{u}(N)\to \mathfrak{su}(N)$.
    It should have the kernel $i\mathbb{R}I$ and descend to a bijective map $\mathfrak{pu}(N)\to \mathfrak{su}(N)$.
    Zeitlin's model is then the Lie--Poisson system on $\mathfrak{g}^* = \mathfrak{pu}(N)^*\simeq \mathfrak{su}(N)$, given by the isospectral flow
    \begin{equation}\label{eq:zeitlin_model}
        \dot W + \frac{1}{\hbar}[W,P] = 0, \qquad -\Delta_N P = W .
    \end{equation}
\end{enumerate}

\subsection{Quantization via representation}\label{sec:quant-via-repr}

To build the quantization maps $T_N$, it is natural to begin from the sphere's symmetry: the Lie group $\operatorname{SO}(3)$.
The connection to the space of functions $C^\infty(\Ss^2)$ is established via its Lie algebra $\mathfrak{so}(3)$ as follows.
Let $x_1,x_2,x_3\in C^\infty(\Ss^2)$ be a choice of Euclidean coordinate functions for the embedded sphere $\Ss^2 \subset \mathbb{R}^3$.
These functions generate $\mathfrak{so}(3)$ as a Lie subalgebra of the Poisson algebra $\mathfrak{g} = (C^\infty(\Ss^2),\{\cdot,\cdot\})$.
But we should think of it differently.
Namely, 
the basis elements $\mathbf{e}_1,\mathbf{e}_2,\mathbf{e}_3 \in \mathbb{R}^3\simeq \mathfrak{so}(3)$ are associated with 
first order differential operators $\mathcal X_\alpha$ on $C^\infty(\mathbb{S}^2)$ given by
\begin{equation*}
    \mathcal X_\alpha\omega = \{x_\alpha, \omega\}, \qquad \text{for}\; \alpha=1,2,3\, .
\end{equation*}
Then $\mathcal X_1,\mathcal X_2,\mathcal X_3$ provides a unitary representation of $\mathfrak{so}(3)$ on $C^\infty(\Ss^2)$ compatible with the Poisson algebra structure.
Indeed, the Jacobi identity yields $\mathcal X_\alpha \mathcal X_\beta - \mathcal X_\beta \mathcal X_\alpha = \{ \{x_\alpha, x_\beta\}, \cdot \}$ and each $\mathcal X_\alpha$ is skew self-adjoint with respect to the $L^2$ inner product.
We now apply the powerful machinery of Schur, Cartan, Weyl, and others (\emph{cf.}~\cite{Kn1996b}) for representations of compact groups and the connection to harmonic analysis.
First, the generators $\mathcal X_1,\mathcal X_2,\mathcal X_3$ can be viewed as the first order elements of the \emph{universal enveloping algebra} $U(\mathfrak{so}(3))$.
Higher order elements correspond to higher order differential operators.
In particular, the Killing form on $\mathfrak{so}(3)$ corresponds to the second order \emph{Casimir element} $\mathcal C = -\sum_\alpha \mathcal X_\alpha^2$, which in this case is the Laplacian $\Delta$ on $\mathbb{S}^2$.
The Casimir element commutes with each $\mathcal{X}_\alpha$.
In particular, if $V \subset C^\infty(\Ss^2)$ is an eigenspace of $\mathcal C=\Delta$ with eigenvalue $\lambda$, then $\mathcal{X}_1, \mathcal{X}_2, \mathcal{X}_3$ provide a sub-representation of $\mathfrak{so}(3)$ on $V$.
Indeed, $\mathcal C[\mathcal X_\alpha,\mathcal X_\beta]V = [\mathcal X_\alpha,\mathcal X_\beta] \mathcal C V = \lambda [\mathcal X_\alpha,\mathcal X_\beta] V$, so $[\mathcal X_\alpha,\mathcal X_\beta] V$ must belong to the eigenspace of $\lambda$, i.e., $[\mathcal X_\alpha,\mathcal X_\beta] V \subset V$.
One can further prove that this representation is irreducible.
In fact, the eigenspaces $V_\ell$ for $\lambda_\ell = -\ell(1+\ell),\ell=0,1,\ldots$ correspond to all the irreducible infinitesimal representations of $\mathrm{SO}(3)$ (i.e., the odd-dimensional representations of $\mathfrak{so}(3))$.
Furthermore, $V_\ell$ has dimension $2\ell+1$ and comes with a canonical basis $Y_{\ell m}$ (the spherical harmonics) corresponding to the \emph{weights} $m=-\ell,\ldots,\ell$ of the representation.

From the point of view just presented, the aim of quantization is to construct a representation of $\mathfrak{so}(3)$ on the Lie algebra $\mathfrak{u}(N)$, which should be `as large as possible': since $\mathfrak{u}(N)\simeq V_0\oplus \cdots\oplus V_{N-1}$ as vector spaces one can hope to capture all the irreducible representations  $V_\ell$ for $\ell=0,\ldots,N-1$.
By identifying the canonical basis $Y_{\ell m}$ in $\mathfrak{u}(N)$, we then obtain the projection $T_N\colon \mathfrak g \to \mathfrak{u}(N)$.
Following Hoppe and Yau~\cite{HoYa1998}, but focusing on geometry rather than formulae,
we now give a recipe for such a construction.

For $N$ even or odd, corresponding to $\ell'=(N-1)/2$, begin with an irreducible, unitary representation $\pi\colon\mathfrak{so}(3)\to \mathfrak{u}(N)$.
Let $S_\alpha = \pi(\mathbf{e}_\alpha)$, chosen so that $S_3$ is diagonal with sorted, equidistant eigenvalues $-\mathrm{i}\ell',\ldots,\mathrm{i}\ell'$ corresponding to the weights of the representation.
Since the representation is unitary, $S_1,S_2,S_3$ share the same spectrum.
Explicitly, $S_1$ and $S_2$ are the tridiagonal matrices given by
\begin{equation*}
    S_1 = \frac{\mathrm{i}}{2}\begin{pmatrix}
        0 &  &  & \\
         & \ddots & f(w)  &\\
        & f(w) & \ddots & \\
        &  &  & 0 \\
    \end{pmatrix}
    \quad\text{and}\quad
    S_2 = \frac{\mathrm{1}}{2}\begin{pmatrix}
        0 &  &  & \\
         & \ddots & f(w)  &\\
        & -f(w) & \ddots & \\
        &  &  & 0 \\
    \end{pmatrix}
\end{equation*}
for $w=-\ell',-\ell'+1,\ldots,\ell'-1$ and $f(w) = \sqrt{\ell'(\ell'+1)-w(w+1)}$.

The Casimir element restricted to the representation is just multiplication by $-\ell'(1+\ell')$ and is therefore given by the diagonal matrix $C = -\sum_{\alpha} S_\alpha^2 = -\ell'(1+\ell')I$.
We now introduce a scaled representation $X_\alpha = \hbar S_\alpha$ in such a way that the matrix $X_\alpha$ correspond to the function $x_\alpha$ on $\mathbb{S}^2$.
Thus, to achieve $\big(\sum_{\alpha} X_\alpha^2\big)^{1/2} = \mathrm{i}I$, corresponding to $\big(\sum_{\alpha} x_\alpha^2\big)^{1/2} = 1$, we take $\hbar = 1/\sqrt{\ell'(1+\ell')} = 2/\sqrt{N^2-1} \eqqcolon \hbar_N$.

From the representation on $\mathbb{C}^N\simeq V_{\ell'}$ we obtain another on $\mathfrak{u}(N)$: for $W\in\mathfrak{u}(N)$, the infinitesimal action of $\mathbf{e}_\alpha$ on $W$ is
\[
    \mathbf{e}_\alpha\cdot W = [S_\alpha, W] = \frac{1}{\hbar}[X_\alpha, W]. 
\]
The Casimir element for this representation is the \emph{Hoppe--Yau Laplacian} $$\Delta_N\colon \mathfrak{u}(N)\to\mathfrak{u}(N)$$ given by 
\[
    \Delta_N = \sum_{\alpha=1}^3 [S_\alpha, [S_\alpha, \cdot ]] = \frac{1}{\hbar^2}\sum_{\alpha=1}^3 [X_\alpha, [X_\alpha, \cdot ]].
\]
Being the Casimir element, any eigenvalue is of the form $\lambda = -\ell(1+\ell)$ for some $\ell \in \mathbb{N}/2$ and the corresponding eigenspace $V_\ell$ has dimension $2\ell + 1$.
From the Peter--Weyl theorem, the representation on $\mathfrak{u}(N)$ decomposes into irreducible, orthogonal components $\mathfrak{u}(N)\simeq V_{\ell_0}\oplus \cdots\oplus V_{\ell_{n-1}}$.
But which $\ell_k$ are included?

Since $\Delta_N \mathrm{i} I = 0$ and $\Delta_N X_\alpha = 2 X_\alpha$, we have $\ell_0 = 0$ and $\ell_1 = 1$.
From general theory, and as desired, $\ell_k = k$ and $n = N$:

\begin{lemma}\label{lem:decomposition}
    $\mathfrak{u}(N)\simeq V_0\oplus \cdots\oplus V_{N-1}$ where $V_\ell$ decomposes orthogonally as
    \begin{equation*}
        V_\ell =\bigoplus_{m=0}^{\ell} V_{\ell m}, \quad \text{with}\quad \operatorname{dim}V_{\ell m} = \left\{ \begin{matrix}
            1 \quad\text{if}\quad m=0 \, \phantom{.}\\
            2 \quad\text{if}\quad m\neq 0 \, . 
        \end{matrix} \right.
    \end{equation*}
    The components $V_{\ell m}$, corresponding to the weights $\pm 2 m$, are mapped to the $\pm m$th diagonals of the matrix in $\mathfrak{u}(N)$.
\end{lemma}

\begin{proof}
    The result is well-known and follows from general theory about multiplicities in the Peter-Weyl decomposition, see for example Kirillov~\cite[ch.~4, particularly sec.~4.9]{Ki2008}.
    We give here a direct, less general proof.

    Consider the space $\mathrm{poly}_N = \{p\in \mathbb{C}[x]\mid \operatorname{deg}p =N\}$ of polynomials.
    A bijection between diagonal matrices $\operatorname{D}(N)$ and $\mathrm{poly}_N$ is given by $p \mapsto p(X_3)$.
    This is an algebra isomorphism, since $(pq)(X_3) = p(X_3)q(X_3)$.
    Let $p_\ell$ denote the Legendre polynomials and consider the corresponding matrices $P_\ell = p_\ell(X_3)$.
    One can readily check that $\Delta_N P_\ell$ is diagonal and corresponds to the polynomial
    \begin{equation*}
        \frac{d}{dx}(1-x^2)\frac{d}{dx} p_\ell(x).
    \end{equation*}
    From Sturm--Liouville theory it then follows that $\Delta_N P_\ell = -\ell(\ell + 1)P_\ell$.
    Thus, we must have that the representation on $\mathfrak{u}(N)$ includes $V_\ell$ for $\ell = 0,\ldots,N-1$.
    By a dimension count, it then follows that $\mathfrak{u}(N)\simeq V_0\oplus \cdots\oplus V_{N-1}$.
    To see that each $V_\ell$ decomposes as stated, one can introduce the ``shift operators'' $X$ and $Y$ obtained as part of the standard basis $H, X, Y$ in the complexification $\mathfrak{sl}(2,\mathbb{C})\simeq \mathfrak{su}(2)_\mathbb{C}$.
    Since $X$ and $Y$ has elements only on the first upper and lower diagonal respectively, it follows from representation theory for $\mathfrak{sl}(2,\mathbb{C})$ that $P_{\ell}\in V_\ell$ gives rise to $2\ell$ new basis vectors $P_{\ell m}\in V_\ell\otimes\mathbb{C}$ for $m=-\ell,\ldots,\ell$ with $P_{\ell 0 } = P_\ell$, each of which is supported on the $m$:th diagonal.
    Another dimension count then concludes that this basis spans $V_\ell\otimes \mathbb{C}$, so the corresponding real basis 
    \begin{equation*}
        \frac{1}{2}\left(P_{\ell m} - P_{\ell (-m)}^\dagger\right) \quad\text{and}\quad \frac{i}{2}\left(P_{\ell m} + P_{\ell (-m)}^\dagger\right)   \quad\text{spans $V_\ell$.} 
    \end{equation*}
    
    %
    %
\end{proof}

\begin{remark}
    For $N$ even, i.e., the $\ell' = (N-1)/2$ ``half-integer spin'' representation on $\mathbb{C}^N$, we still obtain only ``integer spin'' representations on $\mathfrak{u}(N)$.
    This is desirable, since the eigenspaces of the Laplacian on $\Ss^2$ correspond exactly to the integer spin (odd-dimensional) representations of $\mathfrak{so}(3)$.
\end{remark}

The decomposition in Lemma~\ref{lem:decomposition} enables explicit construction of the projection map $T_N\colon \mathfrak{g}\to \mathfrak{u}(N)$.
Indeed, the $\pm m$th diagonals of $W\in\mathfrak{u}(N)$ are spanned by $V_{:m} \coloneqq V_{m m}\oplus\cdots\oplus V_{(N-1)m}\simeq \mathbb{C}^{N-m}$.
In turn, $V_{:m}$ corresponds to the span of the spherical harmonic functions $Y_{mm},\ldots,Y_{(N-1)m}$.
Since $\Delta_N V_{:m} \subseteq V_{:m}$, we obtain the $\pm m$ diagonals by identifying the eigenvector of $\Delta_N|_{V_{:m}}$ with eigenvalue $-(m+k)(1+m+k)$ with the spherical harmonic function $Y_{(m+k)m}$.
Specifically, $\Delta_N|_{V_{:m}}\colon \mathbb{C}^{N-m}\to \mathbb{C}^{N-m}$, as an operator acting on the upper $m$th diagonal of $W$, is a tridiagonal, symmetric matrix.
The spectral decomposition of $\Delta_N|_{V_{:m}}$ can hence be computed by efficient numerical methods of complexity $\mathcal{O}(N)$.
This way, a function $\omega\in\mathfrak{g}$ is projected to the matrix $W=T_N \omega$ from its spherical harmonics coefficients $(\omega_{\ell m})_{\ell \leq N-1}$ in only $\mathcal{O}(N^2)$ numerical operations; see Cifani \emph{et al.}~\cite{CiViMo2023} for details. 
Notice that the $L^2$ adjoint $T_N^*\colon \mathfrak{u}(N)\to \mathfrak{g}$ is a right inverse of $T_N$.
%
%

We have thus arrived at the following:

\begin{theorem}[\cite{HoYa1998}]\label{prop:projectionTN}
    The operators $\Delta_N$ and $T_N$ fulfill
    \begin{itemize}
        \item $T_N\Delta\omega = \Delta_N T_N\omega$;
        \item $T_N 1 = \mathrm{i} I$;
        \item $T_N \{x_\alpha,\omega \} = \frac{1}{\hbar}[X_\alpha,T_N\omega]$.
    \end{itemize}
\end{theorem}

This result achieves the steps 2--4 listed above. Moreover,
(i) the approximation $\Delta_N$ to the Laplacian on $\Ss^2$ is exact on the subspace spanned by spherical harmonics up to order $\ell = N-1$, and (ii) the $\mathrm{SO}(3)$-symmetry of $\Ss^2$ is preserved, since $\mathfrak{so}(3)$ is represented analogously on $\mathfrak{u}(N)$ and $C^\infty(\Ss^2)$.

\begin{remark}
%
%
Zeitlin's original model on the flat torus $\Tt^2$ does
not preserve the translational symmetry generated by the action of the Abelian group $A=\Rr^2$ on $\Rr^2/\Zz^2\simeq \Tt^2$.
The problem is that there is no corresponding action of $A$ on the finite dimensional Lie algebra $\mathfrak{g} = \mathfrak{su}(N)$.
Indeed, whereas there are actions of $\mathbb{R}$ on $\mathfrak{su}(N)$ corresponding to $(a_1,0)\in A$ and $(0,a_2)\in A$, these actions fail to commute.
The numerical consequence is that translation-like motion in Zeitlin's model on $\Tt$ give rise to artifacts and \new{the Gibbs phenomenon}.
\end{remark}

The remaining step (1) above, relating the Poisson bracket $\{\cdot,\cdot\}$ to the scaled commutator $\frac{1}{\hbar}[\cdot,\cdot]$, is at the heart of quantization theory.
The following is a specialization to our setting of a result by Charles and Polterovich~\cite{ChPo2018}.

\begin{theorem}[\cite{ChPo2018}]\label{thm:quantization_limit}
    There exists $\alpha>0$ so for all $\omega\in C^2(\Ss^2)$
    \begin{equation*}
        \lVert \omega \rVert_{L^\infty} - \alpha \hbar \lVert \omega\rVert_{C^2} \leq \sqrt{\frac{N-1}{N+1}} \lVert T_N\omega \rVert_{\infty} \leq \lVert \omega\rVert_{L^\infty},
    \end{equation*}
    where $\lVert W\rVert_\infty = \sup_{\mathbf{u}} \frac{\lvert W\mathbf{u}\rvert}{\lvert\mathbf{u}\rvert}$.
    Further, there exists $\beta>0$ so for all $\omega,\psi\in C^3(\Ss^2)$
    \begin{equation*}
        \lVert \frac{1}{\hbar}[T_N\omega,T_N\psi] - T_N \{\omega,\psi\} \rVert_\infty \leq \beta\hbar \sum_{k=1}^3\lVert\omega\rVert_{C^{k}}\lVert\psi\rVert_{C^{4-k}}.
    \end{equation*}
\end{theorem}

\begin{proof}
    The result was proved for Berezin--Toeplitz operators between holomorphic sections of certain complex holomorphic vector bundles appearing in geometric quantization.
    Up to scaling, the Berezin--Toeplitz quantization is equivalent to the quantization $T_N$ discussed here, as explained in Appendix~\ref{appendix:line_bundles} below.
\end{proof}        

Via a bracket convergence result in the $L^2$ norm (see Lemma~\ref{thm:La-conv} of the appendix), instead of the spectral norm as in Theorem~\ref{thm:quantization_limit}, we obtain convergence of solutions.
\begin{theorem}\label{thm:convergence}
    Let $\omega(t)$ be a solution of \eqref{eq:vorteq2}, and $W(t)$ the solution of \eqref{eq:zeitlin_model} with $W(0) = T_N \omega(0)$. Then, for $t$ fixed, $\lVert T_N^*W(t) - \omega(t) \rVert_{L^2} \to 0$ as $N\to\infty$. 
\end{theorem}

\begin{proof}
    See Appendix~\ref{appendix:convergence_proof}, where we also give a more detailed formulation of the theorem (see Theorem~\ref{thm:space-time-convergence-detailed}), including convergence rates.
\end{proof}

\begin{remark}
    The result in Theorem~\ref{thm:convergence} (and Theorem~\ref{thm:space-time-convergence-detailed}) is analogous to the $L^2$ convergence result of Gallagher~\cite{Ga2002} for Zeitlin's original model on the flat torus~$\mathbb{T}^2$.
    But there are technical differences in the proof: Gallagher's result rely on algebraic multiplication properties of the Fourier basis in $\mathbb{T}^2$, not available on $\Ss^2$.
\end{remark}

In summary, we have a correspondence between functions and matrices that allows transcription of ideas and results from 2-D hydrodynamics to matrix theory.
The dictionary in Table~\ref{tbl:dictionary} contains the most direct examples.

\begin{table}
    \centering
    \renewcommand{\arraystretch}{1.5}
    \begin{tabular}{lcc} \toprule
      & \emph{2-D hydrodynamics} & \emph{matrix hydrodynamics} \\ \midrule
    \emph{Lie group} & $\Phi\in \mathrm{Diff}_\mu(\Ss^2)$ & $F\in\mathrm{PU}(N)$ \\
    \emph{Lie algebra} & $\psi \in C^\infty(\Ss^2)/\mathbb{R}, \;\{\cdot,\cdot\}$ & $P \in \mathfrak{pu}(N), \;\frac{1}{\hbar}[\cdot,\cdot]$ \\
    \emph{Phase space} & $\omega \in C^\infty_0(\Ss^2)$ & $W \in \mathfrak{su}(N)$ \\
    $\mathrm{Ad}^*$ & $\omega\circ\Phi^{-1}$ & $F W F^{-1}$ \\
    $\mathrm{ad}^*$ & $\{\omega,\psi \}$ & $\frac{1}{\hbar}[W,P]$ \\
    \emph{Strong norm} & $\lVert \omega\rVert_{L^\infty} = \operatorname{ess~sup} \omega$ & $\lVert W\rVert_\infty = \sup_{\lvert \boldsymbol{v}\rvert =1} \lvert W\boldsymbol{v}\rvert$ \\
    \emph{Inner product} & $\displaystyle\langle \omega,\omega\rangle_{L^2} = \int_{\Ss^2}\omega^2\mu$ & $\displaystyle \langle W,W\rangle_2 = -\frac{4\pi}{N}\operatorname{tr}(W^2)$ \\ 
    \emph{Casimirs} & $\displaystyle C_f(\omega) = \int_{\Ss^2}f(\omega)\mu$ & $\displaystyle C_f^N(W) = \frac{4\pi}{N}\operatorname{tr}(f(-\mathrm{i}W))$ \\ 
    \emph{Laplacian} & $\displaystyle\Delta \psi = \sum_{\alpha=1}^3 \{ x_\alpha,\{x_\alpha,\psi\}\}$ & $\displaystyle\Delta_N P = \frac{1}{\hbar^2}\sum_{\alpha=1}^3 [ X_\alpha,[X_\alpha,P]]$\\
    \emph{Hamiltonian} & $H(\omega) = -\frac{1}{2}\langle \omega,\Delta^{-1}\omega\rangle_{L^2}$ & 
    $H_N(W) = \frac{2\pi}{N}\operatorname{tr}(W\Delta_N^{-1}W)$
    \\
    \bottomrule
    \end{tabular}
    \caption{Dictionary between 2-D and matrix hydrodynamics.}\label{tbl:dictionary}
\end{table}



%
%


%
%
%
%
%

\section{Coadjoint orbits and their closure}\label{sec:coadjoint-orbits}

For initial vorticity $\omega_0$, consider the orbit $\{S_t(\omega_0)\mid t\geq 0 \}$, where $S_t$ is the flow map for the vorticity equation~\eqref{eq:vorteq2}.
Šverák~\cite{Sv2011} conjectured that orbits for generic $\omega_0\in L^\infty(\Ss^2)$ are not $L^2$ precompact, which can be interpreted as a mathematical formulation of the ``forward enstrophy cascade'' in 2-D turbulence (\emph{cf.}~Kraichnan~\cite{Kr1967}).
Of course, the orbits in Zeitlin's model are always precompact, since they evolve on a finite dimensional sphere.
Nevertheless, simulations with increasing $N$ support Šverák's conjecture.

In Arnold's geometric viewpoint, it follows from equation \eqref{eq:transport_form} that smooth solutions $\omega(t)\in C^\infty(\Ss^2)$ remain on the \emph{coadjoint orbit} $\mathcal{O}(\omega_0) = \{\omega_0\circ\Phi \mid \Phi\in\operatorname{Diff}_\mu(\Ss^2) \}$.
A first step in understanding $\{S_t(\omega_0)\mid t\geq 0 \}$ is thus to characterize $\mathcal{O}(\omega_0)$, a problem solved by
Izosimov, Khesin, and Mousavi~\cite{IzKhMo2016} via the Reeb graph of $\omega_0$ equipped with a measure reflecting the ``density'' of the level sets (see Fig.~\ref{fig:reeb_graph}a).

\begin{theorem}[\cite{IzKhMo2016}]\label{thm:coadjoint_reeb}
    Let $\omega,\psi\in C^\infty(\Ss^2)$ be simple Morse functions (distinct critical values).
    Then $\psi \in \mathcal{O}(\omega)$ if and only if $\omega$ and $\psi$ have the same measured Reeb graph.
\end{theorem}

However, while solutions with smooth initial data remain smooth, the $C^k$-norms for $k\geq 1$ grow fast as the vorticity level-curves quickly become extremely entangled.
This mechanism results in a Reeb graph that is well-defined but futile to track at moderate scales.
It is therefore natural to study the closure of coadjoint orbits in a weaker topology, for instance $L^\infty$ as Yudovich' existence theory suggests.
But the Reeb graph is not stable in $L^\infty$ since it is based on critical values and Morse functions that require $C^2$ topology (see Fig.~\ref{fig:Linf_reeb_breakdown}).
Instead, what persists is the \emph{level-set measure} $\lambda_\omega$, obtained by ``flattening'' the Reeb graph as follows.
Recall the Casimir invariants $C_f(\omega)$.
The mapping $C^\infty(\mathbb{R}) \ni f\mapsto C_f(\omega)$ is a distribution with compact support,
and $\lambda_\omega$ is the corresponding positive Radon measure on $\mathbb{R}$ (see Fig.~\ref{fig:reeb_graph}b). 
The $L^\infty$ closure of a coadjoint orbit then fulfills $\overline{\mathcal{O}(\omega_0)} \subseteq \{ \omega \in L^\infty(\Ss^2)\mid \lambda_\omega = \lambda_{\omega_0}\}$.
(The latter set is characterized in Sec.~\ref{sub:BFR} below.)

\begin{figure}
    \centering
    \includegraphics[width=.4\linewidth]{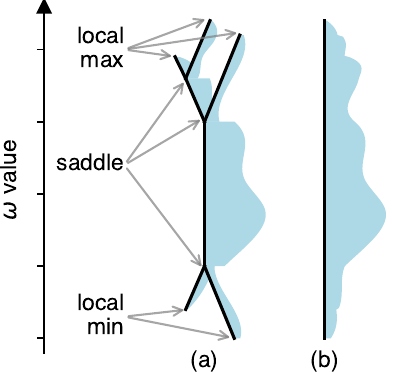}
    \caption{\textbf{(a)} Measured Reeb graph for a function $\omega$ on $\Ss^2$ with three local maxima and two local minima.
    Another function $\psi$ has the same measured Reeb graph if and only if there exists $\Phi\in\operatorname{Diff}_\mu(\Ss^2)$ such that $\psi = \omega\circ\Phi$ (\emph{cf.}~Theorem~\ref{thm:coadjoint_reeb}).\\
    \textbf{(b)} Level-set measure $\lambda_\omega$. It is a flattening, via horizontal projection, of the Reeb graph.
    The measure $\lambda_\omega(I)$ of an interval $I$ is the area of the set $\{x\in \Ss^2\mid \omega(x)\in I \}$.
    }
    \label{fig:reeb_graph}
\end{figure}

\begin{figure}
    \centering
    \begin{subfigure}{0.32\textwidth}
        \centering
        \includegraphics[height=.32\textwidth]{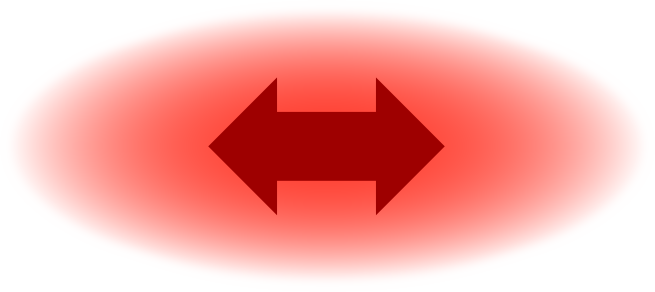}
        \caption{$\omega_0$}
        \label{subfig:Linf_1}
    \end{subfigure}
    \begin{subfigure}{0.32\textwidth}
        \centering
        \includegraphics[height=.32\textwidth]{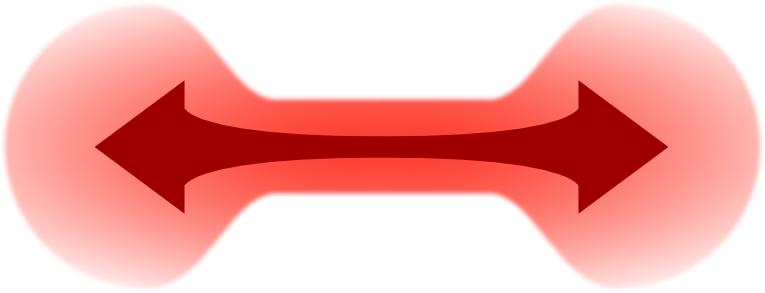}
        \caption{$\omega_0\circ\Phi_k^{-1}$}
        \label{subfig:Linf_2}
    \end{subfigure}
    \begin{subfigure}{0.32\textwidth}
        \centering
        \includegraphics[height=.32\textwidth]{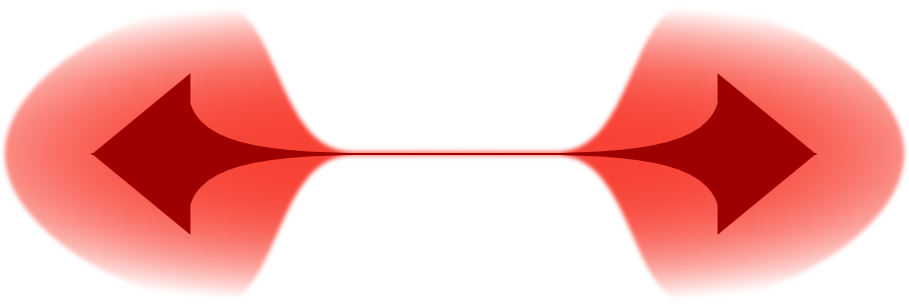}
        \caption{$\lim_{k\to\infty}\omega_0\circ\Phi_k^{-1}$}
        \label{subfig:Linf_3}
    \end{subfigure}
    \caption{Given a vortex configuration $\omega_0$ with one blob, it is possible to find a sequence $\Phi_k$ of area-preserving diffeomorphisms which transports it $\omega_0\circ\Phi_k^{-1}$ such that in the limit $k\to\infty$ it is $L^\infty$-indistinguishable from a configuration with two blobs.
    Thus, the Reeb graph in Theorem~\ref{thm:coadjoint_reeb} is not stable under the $L^\infty$ closure of $\mathcal{O}(\omega_0)$.
    }
    \label{fig:Linf_reeb_breakdown}
\end{figure}

But the $L^\infty$ topology is still too strong: it does not permit \emph{vortex mixing}.
If $\int_{\Ss^2}\omega_0 = 0$, one can imagine a sequence $\Phi_k$ of increasingly entangling diffeomorphisms for which $\lim_{k\to\infty}\int_{\Omega}\omega_0\circ\Phi_k = 0$ over any open subset $\Omega\subset\Ss^2$.
Yet, however complicated the sequence is, we always have $\lVert \omega_0\circ\Phi_k\rVert_{L^\infty} = \lVert \omega_0\rVert_{L^\infty}$.
Shnirelman~\cite{Sh1993} suggested instead the $L^\infty$ weak-$*$ topology, for which $\omega_0\circ\Phi_k \overset{\ast}{\rightharpoonup} 0$ is possible. 
Furthermore, for a generic $\omega_0 \in L^\infty(\Ss^2)$, the weak-$*$ closure $\overline{\mathcal{O}(\omega_0)}^*$ 
is convex and can be explicitly characterized~\cite{Sv2011, DoDr2022a}.
Importantly, the Casimirs $C_f$ are not weak-$*$ continuous (except for $f(z) = z$), so elements in $\overline{\mathcal{O}(\omega_0)}^*$ can have different level-set measures, which signifies vortex mixing.

Matrix hydrodynamics offers a concrete way to address the weak-$*$ topology in 2-D turbulence, as now explained.



The matrix version of the level-set measure $\lambda_\omega$ is the \emph{empirical spectral measure} from random matrix theory.
Indeed, $W = T_N(\omega)$ gives rise to the atomic measure $\lambda_W^N = \frac{1}{N}\sum_{k=1}^N \delta_{\lambda_k}$ on $\mathbb{R}$, where $\mathrm{i}\lambda_1,\ldots,\mathrm{i}\lambda_N$ are the eigenvalues of $W$.

\begin{theorem}\label{thm:spectral_measure_convergence}
    Let $\omega \in C^2(\Ss^2)$.
    Then $\lambda_{T_N(\omega)}^N \rightharpoonup \lambda_\omega$ as $N\to\infty$.
\end{theorem}

\begin{proof}
    Let $I = [\inf \omega, \sup \omega]$ and $f\in C^\infty(I)$.
    We need to prove that $C_f^N(T_N(\omega))\to C_f(\omega)$, which is equivalent to weak convergence of the measure.
    
    We can restrict $f$ to be polynomials, which are dense in $C^\infty(I)$.
    Since $C_f(\omega) = \sum_m f_m C_{z^m}(\omega)$ and similarly for $C^N_f$, it is enough to prove that $C^N_{z^m}(T_N\omega)\to C_{z^m}(\omega)$ for $m=0,1,\ldots$.
    Set $W_N = T_N\omega$. Then
    \begin{align*}
        & \frac{1}{4\pi}C^N_{z^m}(W_N) = \frac{(-i)^m}{N}\operatorname{tr}(W_N^m) = \frac{(-i)^m}{N}\operatorname{tr}(W_N W_N^{m-1}) = \\
        & \frac{(-i)^m}{N}\operatorname{tr}\left( W_N\big(W_N^{m-1}-T_N(\omega^{m-1}) + T_N(\omega^{m-1})\big) \right) = \\
        &  \frac{(-i)^{m-1}}{N}\operatorname{tr}\left( W_N T_N(\omega^{m-1})\right) + \operatorname{tr}\left(W_N\big(W_N^{m-1}-T_N(\omega^{m-1})\big) \right).
    \end{align*}
    Since the scaled Frobenius inner product on $\mathfrak{u}(N)$ corresponds to the $L^2$ inner product, i.e., $\langle T_N^*A,T_N^*B \rangle_{L^2} = -\frac{4\pi}{N}\operatorname{tr}(AB)$, we have 
    \begin{equation*}
        \lim_{N\to\infty} \frac{(-i)^m}{N} \operatorname{tr}(W_N, T_N(\omega^{m-1})) = \frac{1}{4\pi}\langle \omega,\omega^{m-1}\rangle_{L^2} = \frac{1}{4\pi}C_{z^m}(\omega).    
    \end{equation*}
    We also have
    \begin{multline*}
        \left| \frac{(-i)^m}{N}\operatorname{tr}\left( W_N\big(W_N^{m-1}-T_N(\omega^{m-1})\big)\right)  \right| \leq \\ \underbrace{\left(\frac{1}{N}\sum_{k=1}^N \lvert \lambda_k \rvert \right)}_{\lVert W_N\rVert_1}\lVert W_N^{m-1} - T_N(\omega^{m-1}) \rVert_\infty \to 0,
    \end{multline*}
    which follows from Lemma~\ref{lem:power_quantization_convergence} below.
    \end{proof}
    
    \begin{lemma}\label{lem:power_quantization_convergence}
        Let $\omega\in C^2(\Ss^2)$. Then $\lVert T_N(\omega)^m - T_N(\omega^m)\rVert_\infty \to 0$ as $N\to\infty$.
    \end{lemma}
    
    \begin{proof}
        The result is true for $m=1$.
        Assume it for $m-1$. With $W_N = T_N(\omega)$ we have
        \begin{align*}
            & \lVert W_N^m - T_N(\omega^m)\rVert_\infty = \\
            & \lVert W_N W_N^{m-1} - T_N(\omega^m) + W_N T^N(\omega^{m-1}) - W_N T_N(\omega^{m-1})\rVert_\infty \leq \\
            &\lVert W_N (W_N^{m-1} - T_N(\omega^{m-1}))\rVert_\infty + \\
            &\lVert T_N(\omega\omega^{m-1}) - W_N T^N(\omega^{m-1}) \rVert_\infty \leq \\
            &\lVert W_N \rVert_1 \lVert W_N^{m-1} - T_N(\omega^{m-1})\rVert_\infty + \\
            &\lVert T_N(\omega\omega^{m-1}) - W_N T^N(\omega^{m-1}) \rVert_\infty .
        \end{align*}
        That the first term vanishes as $N\to\infty$ follows from the assumption.
        That the second term vanishes follows from product estimates in Berezin-Toeplitz quantization theory \cite[prop.\,3.12]{ChPo2018} and its equivalence to quantization via representations, as described in Appendix~\ref{appendix:line_bundles}.
    \end{proof}

The level-set measure $\lambda_\omega$ is invariant under the Euler flow \eqref{eq:vorteq2}.
Likewise, $\lambda^N_W$ is invariant under the Euler--Zeitlin flow \eqref{eq:zeitlin_model}.
In particular, for any $f\in C^\infty(\mathbb{R})$ the function $C^N_f(W) = \int_{\mathbb{R}}f \, d\lambda^N_W$ is a Casimir for $\mathfrak{su}(N)$.
A direct corollary of Theorem~\ref{thm:spectral_measure_convergence} is that $C^N_f(T_N(\omega))\to C_f(\omega)$ as $N\to \infty$.
In this sense, the Euler--Zeitlin equation nearly conserves all Casimirs of the vorticity equation.

Since the coadjoint action of $F\in \mathrm{SU}(N)$ on $W\in\mathfrak{su}(N)$ is $F W F^\dagger$, it follows from 
the spectral theorem 
that the coadjoint orbit $\mathcal{O}^N(W) = \{ F W F^\dagger\mid F\in \mathrm{SU}(N)\}$ is the set of skew-Hermitian matrices 
with the same spectrum as $W$, or, equivalently, the same empirical measure $\lambda^N_W$.

It is instructive to compare with coadjoint orbits for $\operatorname{Diff}_\mu(\Ss^2)$: the finite-dimen\-sional coadjoint orbit $\mathcal{O}^N(W)$ corresponds (see Proposition~\ref{prop:orbit_thm_reformulated} below) to the set $$\overline{\mathcal{O}}(\omega_0) \coloneqq \{  \omega \in L^\infty(\Ss^2)\mid \lambda_\omega = \lambda_{\omega_0}\}.$$
This set is strictly larger than the smooth orbit $\mathcal{O}(\omega_0)$ (characterized by its measured Reeb graph), but strictly smaller than Shnirelman's weak closure $\overline{\mathcal{O}(\omega_0)}^*$ (characterized by convex vortex mixing, \emph{cf.}~\cite{DoDr2022a}).
For a fixed $N$, matrix hydrodynamics thereby behaves as the Euler equations in the strong $L^\infty$ topology: Casimirs are exactly conserved.
But we can also regard all $N$ simultaneously, which gives a hierarchy of increasingly refined topologies.



\begin{definition}\label{def:zeitlin_topology}
    Let $\mathcal B_r = \{ \omega \in L^\infty(\Ss^2)\mid \lVert\omega\rVert_{L^\infty} \leq r \}$.
    The \emph{matrix topology} on $\mathcal B_r$ is the Fréchet topology defined by the semi-norms
    \[
        \lvert \omega\rvert_N \coloneqq \lVert T_N\omega \rVert_\infty .
    \]
\end{definition}

The semi-norm $\lvert \cdot\rvert_N$ can ``see'' structures down to the length scale $\hbar_N$.
Furthermore, each length scale $\hbar_N$ gives an associated truncated level-set measure $\lambda_{T_N(\omega)}^N$.
To study how it varies, for fixed $N$ as $t\to\infty$, gives information about vortex mixing relative to the scale $\hbar_N$ and a connection to random matrix theory (see Fig.~\ref{fig:mixing_measure}).


\begin{figure}
    \centering
    \includegraphics[width=.4\linewidth]{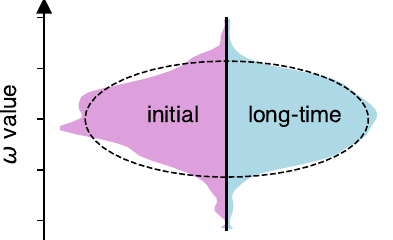}
    \caption{The simulation $W(t)\in \mathfrak{su}(512)$ in Fig.\,\ref{fig:non-vanishing} is projected to $\bar W(t) \in \mathfrak{su}(480)$. The empirical spectral measure $\lambda^{480}_{\bar W(t)}$ is then displayed at the initial (left) and  final (right) time.
    We observe some tendency towards the Wigner semicircle distribution (dashed curve), except the rims survive due to vortex blobs (compare Fig.\,\ref{fig:non-vanish-final}).
    }
    \label{fig:mixing_measure} 
\end{figure}

The following result connects matrix hydrodynamics to Shnirelman's notion of weak-$*$ vortex mixing.

\begin{lemma}\label{lem:topology_strengths}
    The matrix topology is equivalent to the weak-$*$ topology on $\mathcal B_r$.
\end{lemma}

\begin{proof}
    The set $V = \{f \mid \exists N : f \in T_N^*\mathfrak{su}(N) \}$ is dense in $L^1(\Ss^2)$.
    For $\omega \in L^\infty(\Ss^2)$ and a test function $f = T_N^*W \in V$ we have
    \begin{equation*}
        \lvert \langle \omega, f \rangle \rvert
        = \lvert \langle\omega , T_{N}^* W \rangle \rvert 
         = \frac{4\pi}{N}\lvert \operatorname{tr}(W T_{N}\omega) \rvert \leq \frac{4\pi}{N} \lvert \omega \rvert_{N} \lVert W \rVert_1. 
    \end{equation*}
    Thus, if, for any $N$, $\lvert \omega_k\rvert_N \to 0$ as $k\to\infty$, then $\langle \omega_k , f \rangle \to 0$ for all $f\in V$, which implies weak-$*$ convergence since the sequence $\omega_k$ in $\mathcal{B}_r$ is bounded in $L^\infty$-norm.
    On the other hand, assuming that $\omega_k \rightharpoonup^{*} 0$, it follows that each spherical harmonic coefficient $(\omega_k)_{\ell m}\to 0$ as $k\to \infty$.
    Thus, since $T_N\omega$ is expanded in a finite basis with the spherical harmonic coefficients, it follows that $\lVert T_N\omega_k \rVert_\infty \to 0$ as $k\to \infty$.
\end{proof}

\begin{remark}
    Alternatively, one can view Lemma~\ref{lem:topology_strengths} as a special case of the general principle that two Hausdorff topologies on a set are equivalent if the first is compact and the second is weaker than the first.
\end{remark}

In consequence, a plausible relaxation of Euler's equations \eqref{eq:vorteq2} to the weak-$*$ topology is to consider the Euler--Zeitlin equation \eqref{eq:zeitlin_model} \emph{simultaneously for all $N \geq N_0$}.
Via the mappings $T_{N-1}\circ T^*_N\colon \mathfrak{u}(N)\to \mathfrak{u}(N-1)$ one can consider an inverse limit $\mathfrak{g} = \underleftarrow{\lim} \, \mathfrak{u}(N)$.
But this construction is too rigid, since $\mathfrak{g}$ is not invariant under the simultaneous Euler--Zeitlin flow.
A plausible direction is to relax the inverse limit, for example based on the equivalence $\{W'_i\}_{i=N_0}^\infty \sim \{W_i\}_{i=N_0}^\infty$ if $\lim_{i\to\infty}\lVert W'_i-W_i\rVert_\infty = 0$, and attempt to prove that the simultaneous Euler--Zeitlin flow is well-defined on equivalence classes.
For such a construction, global existence is built-in (albeit not uniqueness).








\subsection{Schur-Horn-Kostant convexity}\label{sub:BFR}

In their paper on Schur-Horn-Kostant convexity for the area-preserving diffeomorphism group of the annulus, Bloch, Flaschka, and Ratiu~\cite{BlFlRa1993} (BFR) write: ``Diffeomorphism groups are huge, infinite-dimensional Lie groups, but in some respects they can be strikingly similar to finite-dimensional semisimple groups.''
Matrix hydrodynamics can be viewed as an actualization of this statement.

Consider the annulus (or cylinder)
\begin{equation*}
    \mathbb{A} = \mathbb{S}^1 \times [-1,1] = \{(\exp(\mathrm{i}\phi), z)\mid 0\leq \phi< 2\pi, -1\leq z \leq 1  \}
\end{equation*}
equipped with the area-form $\nu = d\phi\wedge d z$.
Let $\mathrm{Diff}_\nu(\mathbb{A})$ denote the \new{infinite-dimensional} Lie group\footnote{\new{More precisely, $\mathrm{Diff}_\nu(\mathbb{A})$ is an infinite-dimensional Lie group in the category of Fréchet manifolds (\emph{cf.}~Hamilton~\cite[Thm.~2.5.3]{Ha1982}).}} of area-preserving diffeomorphisms.
Its Lie algebra $\mathfrak{a}$ can be identified with the space of smooth Hamiltonian functions on $\mathbb{A}$, modulo constants, that are constant on the edges $z=-1$ and $z=1$ (corresponding to tangential boundary conditions), equipped with the Poisson bracket induced by $\nu$.
The point of BFR is to think \emph{morally} of the maps $(\phi,z)\mapsto (\phi + \alpha(z), z)$ as the maximal Abelian subgroup $\mathrm{D}\subset \mathrm{Diff}_\nu(\mathbb{A})$ corresponding to diagonal matrices in the matrix case.
The ``projection onto the diagonal'' of a Lie algebra element is then given by averaging
\begin{equation*}
    \pi\colon \mathfrak{a}\ni \varphi \mapsto \int_{-1}^1 \varphi(\phi,\cdot)d\phi \in \mathfrak{d} = T_e \mathrm{D}.
\end{equation*} 
The corresponding ``Weyl group'' $\mathrm{WG}$ consists of maps $(\phi,z)\mapsto (j(z)\phi, a(z))$, where $a:[-1,1]\to [-1,1]$ is invertible, measure preserving, and $j\colon [-1,1]\to \{-1,1 \}$ \new{is measurable}.
In the smooth category, $\mathrm{WG}$ has only two elements, with $a(z)=z$ and $j(z)=\pm 1$.
However, the idea of BFR is to work in a weaker topology, such that the Weyl group can ``rearrange'' the level-sets of functions in $\mathfrak{d}$ corresponding to changing the order of elements of diagonal matrices.

Let $\overline{\mathrm{Diff}_\nu(\mathbb{A})}$ denote the semigroup of measure preserving endomorphisms, which is the completion of $\mathrm{Diff}_\nu(\mathbb{A})$ with respect to the strong operator topology obtained by thinking of $\Phi \in \mathrm{Diff}_\nu(\mathbb{A})$ as the unitary operator on $L^2(\mathbb{A})$ given by $\omega \mapsto \omega\circ\Phi$.
Let $\overline{\mathrm{WG}}$ denote the corresponding completion of the Weyl group, which  \emph{does} allow rearranging.

The first result of BFR is the \emph{convexity theorem}, which is a beautiful, infinite-dimensional analog of Schur's theorem for matrices.

\begin{theorem}[Convexity theorem~\cite{BlFlRa1993}]\label{thm:convexity_theorem}
    Let $\zeta\in L^2(\mathbb{A})$ be a bounded, non-increasing, right continuous function of $z$ alone, and \new{let} $\overline{\mathcal{O}}(\zeta) = \zeta\circ \overline{\mathrm{Diff}_\nu(\mathbb{A})}$ \new{be} its orbit.
    Then $\pi(\overline{\mathcal{O}}(\zeta))\subset L^2([-1,1])$ is a weakly compact, convex set.
    Furthermore, the set of extreme points of $\pi(\overline{\mathcal{O}}(\zeta))$ is the orbit $\zeta\circ \overline{\mathrm{WG}}$.        
\end{theorem}

The second result of BFR is the \emph{orbit theorem}, which is a beautiful, infinite-dimensional analog of Horn's theorem for matrices and the result that two (skew-)Hermitian matrices $A$ and $B$ share the same spectrum if and only if $\operatorname{tr}(A^m) = \operatorname{tr}(B^m)$ for $m=1,\ldots,N$.

\begin{theorem}[Orbit theorem~\cite{BlFlRa1993}]\label{thm:orbit_thm}
    If $\omega \in L^\infty(\mathbb{A})$, there is a unique bounded, non-increasing, and right continuous function $\zeta$ of $z$ alone such that $\omega\in \overline{\mathcal{O}}(\zeta)$.
    Furthermore, the orbit $\overline{\mathcal{O}}(\zeta)$ consists of all functions $\sigma\in L^\infty(\mathbb{A})$ with the same Casimirs as $\omega$, i.e.,
    \begin{equation*}
        \int_{\mathbb{A}} \sigma^m \, \nu = 2\pi \int_{-1}^1 \zeta^m \, d\phi 
        , \quad m=1,2,\ldots . 
    \end{equation*}        
\end{theorem}

These two results are directly related to the question of completion of coadjoint orbits, flattened Reeb graphs, and vortex mixing as discussed above.
Furthermore, the convexity theorem is directly related to mixing operators and the canonical splitting presented in Sec.~\ref{sec:splitting-and-mixing} below.

We now translate the results of BFR to the spherical setting.
Archimedes demonstrated that the horizontal projection of the cylinder to the sphere is area-preserving (i.e., symplectic).\footnote{For this reason, V.~Arnold called Archimedes ``the first symplectic geometer''.}
Explicitly, it is the map $\alpha\colon \mathbb{A}\to \Ss^2$ given by 
\begin{equation*}
    \alpha\colon (\phi,z)\mapsto (\sqrt{1-z^2}\cos\phi,\sqrt{1-z^2}\sin\phi, z) .    
\end{equation*}
It is almost everywhere invertible, with 
\begin{equation*}
    \alpha^{-1}\colon (x_1,x_2,x_3) \mapsto (\arg(x_1 + \mathrm{i} x_2), x_3).
\end{equation*}
Notice that smooth functions in $\mathfrak{a}$ are constant exactly on the set-valued points $$\alpha^{-1}(0,0,\pm 1) = \big(\Ss^1,\pm 1\big).$$
Since it is measure preserving $\alpha^*\mu = \nu$, it follows that $\overline{\mathrm{Diff}_\nu(\mathbb{A})}$ and $\overline{\mathrm{Diff}_\mu(\Ss^2)}$ are isomorphic via 
\begin{equation*}
    \overline{\mathrm{Diff}_\mu(\Ss^2)} = \alpha\circ \overline{\mathrm{Diff}_\nu(\mathbb{A})}\circ\alpha^{-1} .
\end{equation*}
Furthermore, for $p \in \{1,\ldots,\infty\}$ any $f \in L^p(\Ss^2)$ is mapped to $f\circ\alpha \in L^p(\mathbb{A})$ and vice versa.
Thus, the convexity theorem and the orbit theorem translate to the spherical setting, such that ``diagonal'' functions are zonal (i.e., constant along longitudes).

The orbit theorem provides a natural closure of the smooth coadjoint orbits:


\begin{proposition}\label{prop:orbit_thm_reformulated}
    Let $\omega_0\in L^\infty(\Ss^2)$ and let $\zeta \in L^\infty(\Ss^2)$ be its corresponding unique non-increasing, right continuous zonal function (obtained via the orbit theorem).
    Then 
    \begin{equation*}
        \overline{\mathcal{O}}(\omega_0) \coloneqq \{ \omega\in L^\infty(\Ss^2)\mid \lambda_\omega = \lambda_{\omega_0}\} = \zeta\circ\overline{\mathrm{Diff}_\mu(\Ss^2)}. 
    \end{equation*}
\end{proposition}

\begin{proof}
    This result is essentially a reformulation of the orbit theorem.
    It follows since $\lambda_\omega = \lambda_{\omega_0}$ if and only if $\omega_0$ and $\omega$ have the same polynomial Casimirs (as already used in the proof of Theorem~\ref{thm:spectral_measure_convergence} above).
\end{proof}

The closure $\overline{\mathcal{O}}(\omega_0)$ in Proposition~\ref{prop:orbit_thm_reformulated} is located strictly in between the smooth coadjoint orbits $\mathcal{O}(\omega_0)$ and Shnirelman's weak closure $\overline{\mathcal{O}(\omega_0)}^*$, i.e., $$\mathcal{O}(\omega_0)\subset \overline{\mathcal{O}}(\omega_0) \subset \overline{\mathcal{O}(\omega_0)}^*.$$
In a sense, $\overline{\mathcal{O}}(\omega_0)$ is the smallest closure that allows the matrix theory results to remain intact.
It is also compatible with Yudovich's global well-posedness theory, which gives a transport map $\Phi_t\in \overline{\mathrm{Diff}_\mu(\Ss^2)}$ for initial data $\omega_0\in L^\infty(\Ss^2)$.

\begin{remark}
    Bloch, Flaschka, and Ratiu~\cite{BlFlRa1993} made a connection between matrices and area-preserving diffeomorphisms of the annulus via the Toda lattice model \cite{To1970} and its continuum limit (the dispersionless Toda equations, \emph{cf.}~\cite{BrBl1990}).
    However, that model is somewhat unfulfilling, since the matrices are tridiagonal and do not ``fill out'' the infinite-dimensional group as they grow in size.\footnote{The authors mention this and hint at a series corresponding to Hermitian matrices that ``seem reasonable to consider'': they rediscovered quantization.}
    Matrix hydrodynamics offers a remedy: a model that gives a direct link between the matrix groups $\mathrm{SU}(N)$ and area-preserving diffeomorphisms.
    In addition to Hamiltonian systems,
    the link also leads to matrix flow analogs of incompressible porous medium equations~\cite{KhMo2023}, via the gradient flow interpretation of the Toda lattice.
\end{remark}

\section{Canonical splitting and mixing operators} \label{sec:splitting-and-mixing}

We now showcase a technique in 2-D turbulence enabled via matrix hydrodynamics: \emph{canonical scale separation} \cite{MoVi2022}.

Stationary solutions of the Euler--Zeitlin equations~\eqref{eq:zeitlin_model} are characterized by the relation $[P,W]=0$.
The \emph{stabilizer} of $P$ is the subspace $\stab_P = \{ W\in  \mathfrak{su}(N)\mid [P,W]=0\}$.
For a state $W$ and $P= -\Delta_N^{-1}W$, the distance of $W$ to $\stab_{P}$ thus measures how far from stationary the state is.

Let $\Pi_P$ denote the projection onto $\stab_P$ with respect to the inner product $\langle W, U\rangle_2 = \frac{2\pi}{N}\operatorname{tr}(W^\dagger U)$ corresponding to $L^2$.
It induces a canonical splitting $W = W_s + W_r$ where $W_s = \Pi_P W$.
Explicitly, if $E\in \mathrm{U}(N)$ is an eigen basis for $P$, then $\Pi_P(W) = E^*\pi_N(EWE^*)E$, where $\pi_N$ is the orthogonal projection onto the space of diagonal matrices (i.e., extracting the diagonal).
From this formula, we can express the Euler--Zeitlin equations in the (now time dependent) variables $W_s$ and $W_r$:
\begin{equation}\label{eq:EZ_splitting}
\begin{array}{ll}
&\dot{W}_s = [B,W_s] - \Pi_P[B,W_r]\\
&\dot{W}_r = [P,W_r] - [B,W_s] + \Pi_P[B,W_r],
\end{array}
\end{equation}
where $B$ is the unique element in the complement $\stab_P^\perp$ such that $$[B,P]=\Pi_P^\perp\Delta^{-1}[P,W].$$
(See Modin and Viviani~\cite[sect.~3.1]{MoVi2022} for a derivation of the equations~\eqref{eq:EZ_splitting}.)
By construction, the decomposition is orthogonal in $\langle\cdot,\cdot\rangle_2$.
Curiously, it also fulfills a ``reversed'' orthogonally condition with respect to the energy norm $\lVert W \rVert_{E}^2\coloneqq \langle W, \Delta_N^{-1}W\rangle_{2}$.
 
\begin{lemma}[\cite{MoVi2022}]\label{lem:orthogonal_energy}
    The canonical splitting fulfills 
    \begin{align*}
        \lVert W \rVert_2^2 = \lVert W_s \rVert_2^2 + \lVert W_r \rVert_2^2 \quad \text{and}\quad
        \lVert W_s \rVert_{E}^2 = \lVert W \rVert_{E}^2 + \lVert W_r \rVert_{E}^2 .        
    \end{align*}
\end{lemma}

\begin{remark}
    In the special case when $P$ is diagonal, the projection $\Pi_{P}W$ selects the diagonal of $W$. 
    From Schur's theorem it then follows that $C_f^N(\Pi_{P}W) \leq C_f^N(W)$ for convex functions $f$, and that $C_f^N(\Pi_{P}W) = C_f^N(W)$ for all convex $f$ if and only if $P$ and $W$ are simultaneously diagonalizable.
    Since a general $P\in \mathfrak{u}(N)$ can be diagonalized, and since $(P,W)\mapsto \Pi_P W$ is equivariant under the adjoint action, we also have $C_f^N(\Pi_{P}W) \leq C_f^N(W)$ for generic~$P$.
\end{remark}


Why is the canonical splitting interesting?
The theory of statistical mechanics, applied to Euler's equations on $\Ss^2$, predicts a long-time behavior of small scale fluctuations about a large scale, steady state~\cite{He2013}.
Thus, we expect $W_r = W - W_s$ to decrease\footnote{But due to Poincaré recurrence only in a statistical sense: see \href{https://klasmodin.github.io/blog/2023/zeitlin-reversibility/}{zeitlin-reversibility-blog} \cite{Mo2023-reversibility-blog}.} with time in the energy norm, which is in-sensitive to small scales.
Indeed, in numerical simulations the component $W_s$ favors evolution into large scales, while $W_r$ favors small scales (see Fig.~\ref{fig:non-vanishing_canonical} and \cite[Fig.\,3b]{MoVi2022}). 
This canonical separation of scales is a perspective on Kraichnan's~\cite{Kr1967} backward energy and forward enstrophy cascades.

\begin{figure}
    \centering
    \begin{subfigure}{0.49\columnwidth}
        \includegraphics[width=\textwidth]{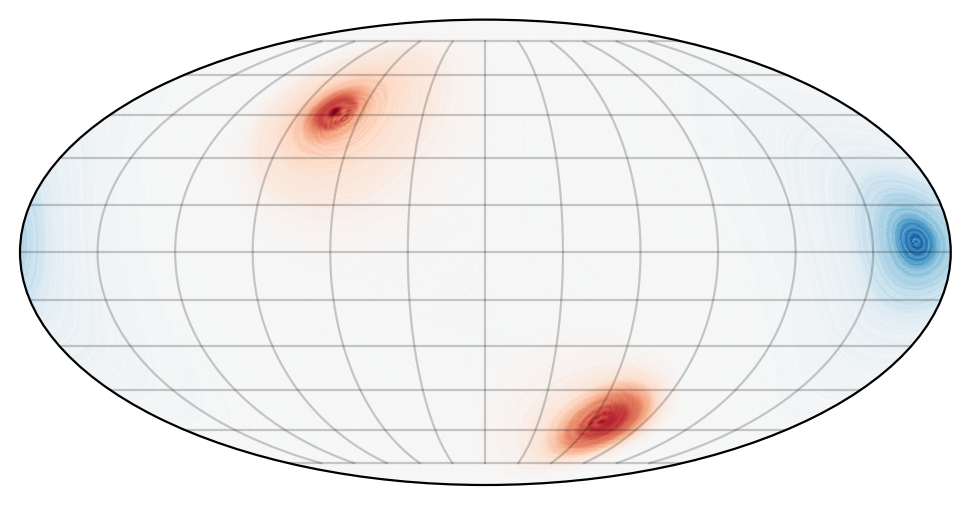}
        \caption{component $W_s$}
        \label{fig:non-vanish-final_s}
    \end{subfigure}
    \begin{subfigure}{0.49\columnwidth}
        \includegraphics[width=\textwidth]{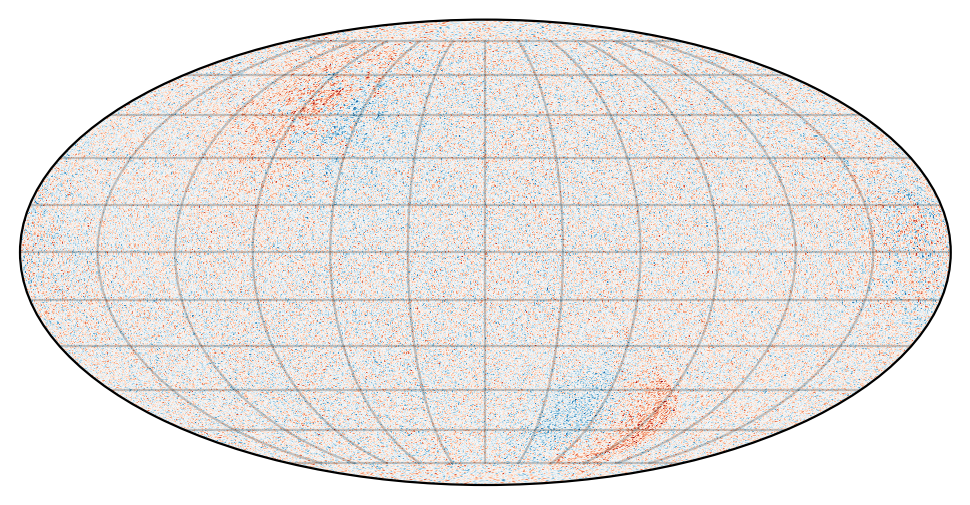}
        \caption{component $W_r$}
        \label{fig:non-vanish-final_r}
    \end{subfigure}
    \caption{
        Canonical splitting $W=W_s + W_r$ of the long-time state $W$ from Fig.~\ref{fig:non-vanish-final}.
        The components capture the large and small scales.
        The projection $W\mapsto W_s$ is a mixing operator.
    }
    \label{fig:non-vanishing_canonical}
\end{figure}





The canonical splitting came naturally from matrix Lie theory via the stabilizer.
We now carry over the developments from $\mathfrak{su}(N)$ to $L^\infty(\Ss^2)$, which gives a connection to \emph{mixing operators} studied by Shnirelman~\cite{Sh1993}.
They are the bistochastic operators, defined in terms of kernels as
\begin{equation*}
    \mathcal K = \{ k\colon \Ss^2\times \Ss^2 \to \mathbb{R} \mid k\geq 0, \int_{\Ss^2} k(x,\cdot) = \int_{\Ss^2} k(\cdot, y) = 1 \} . 
\end{equation*}
Given $\omega,\omega' \in \overline{\mathcal O(\omega_0)}^*\cap \{ H(\omega) = H(\omega_0)\}$, the partial ordering $\omega' \preceq \omega$ means there exists $K \in \mathcal K$ such that $\omega' = K \omega$.
Shnirelman defines $\omega^*$ as a \emph{minimal flow} if $\omega\preceq \omega^*$ implies $\omega^*\preceq \omega$, and he proves that such $\omega^*$ has the stationary property $\omega^* = f\circ \psi^*$ for a bounded, monotone function $f$.

The analog of $\Pi_P$ is the operator $\Pi_\psi$ for $L^2$ projection onto $\mbox{stab}_\psi = \overline{\operatorname{ker}\{\psi,\cdot \}}^{L^2}$, 
formally given by averaging along the level-sets of $\psi$.
For a generic $\psi\in C^1(\Ss^2)$, it is well-defined as a norm-1 operator $\Pi_\psi\colon L^p(\Ss^2)\to L^p(\Ss^2)$ \cite[prop.\,5.4]{MoVi2022}.

\begin{remark}
    Notice that the special case $\Pi_{x_3}$ is the projection $\pi$ in the Convexity theorem~\ref{thm:convexity_theorem} of Bloch, Flaschka, and Ratiu~\cite{BlFlRa1993}. 
    If $\zeta$ is the zonal function from the Orbit theorem, so that $\psi = \zeta\circ\Phi$, and the measure preserving map $\Phi$ happens to be invertible, then we have the formula $\Pi_\psi(\omega) = \pi(\omega\circ\Phi^{-1})\circ\Phi$.
    From this perspective, we expect that $C_f(\Pi_\psi(\omega)) = C_f(\omega)$ for all convex functions $f$ if and only if $\psi$ and $\omega$ are ``simultaneously diagonalizable'', i.e., $\omega\circ\Phi^{-1}$ is zonal.
    This viewpoint is related to the work by Dolce and Drivas~\cite{DoDr2022a}, which characterizes Shnirelman's minimal flows via convex Casimirs.

    
    
        
\end{remark}


\begin{proposition}[\cite{MoVi2022}]\label{prop:canonical_mixing}
    $\Pi_\psi$ is a mixing operator.
\end{proposition}


In the canonical splitting $\omega = \omega_s + \omega_r$, the observed numerical decay of the energy norm $\lVert \omega_r \rVert_E^2 = \langle \omega_r,\Delta^{-1}\omega_r\rangle_{L^2}$ is compatible with Shnirelman's weak-$*$ convergence:


\begin{proposition}\label{prop:Wr_convergence}
    Let $\omega(t)\in L^\infty(\Ss^2)$ be a global solution of \eqref{eq:vorteq2} and consider its canonical splitting $\omega(t)= \omega_s(t) + \omega_r(t)$.
    Then $\|\omega_r(t_n)\|_{E}\rightarrow 0$ if and only if $\omega_r(t_n)\overset{\ast}{\rightharpoonup} 0$ as $t_n\to\infty$.
\end{proposition}

\begin{lemma}\label{prop:lemma1}
    Let $\lbrace\omega_n\rbrace$ be a uniformly bounded sequence in $L^\infty$.
    Then, for the limit $n\to\infty$, the following statements are equivalent:
    \begin{enumerate}
        \item $\|\omega_n-\overline{\omega}\|_{E}\rightarrow 0$,
        \item $\omega_n\overset{\ast}{\rightharpoonup} \overline{\omega}$  weakly-$*$ in $L^\infty$,
        \item $\omega_n\rightharpoonup\overline{\omega}$ weakly in $L^2$.
    \end{enumerate}
\end{lemma}
\proof Weak-$*$ convergence in $L^\infty$ implies weak convergence in $L^2$, since $L^2\subset L^1$, for bounded domains. 

On the other hand, if $\omega_n\rightharpoonup\overline{\omega}$ weakly in $L^2$, then being  $L^2$ compactly embedded in $H^{-1}$ by the Rellich--Kondrachov theorem, we have that $\|\omega_n-\overline{\omega}\|_{E}\rightarrow 0$.

Finally, it remains to prove that convergence in $H^{-1}$ implies convergence in the weak-$*$ topology.
Let $\phi\in C^1$, then we have that
\[
|\int (\omega_n-\overline{\omega})\phi dS| \leq \|\omega_n-\overline{\omega}\|_{E}\|\phi\|_1\rightarrow 0, \mbox{ for }n\rightarrow\infty,
\]
and, by hypothesis,
\[\|\omega_n-\overline{\omega}\|_\infty\leq C,\]
for any $n\geq 1$.
Hence, since $C^1$ is dense in $L^1$ and the sequence of linear functionals $f_n(\phi) := \int (\omega_n-\overline{\omega})\phi dS$ defined on $L^1$ is bounded in norm by the constant $C$ and so they are equicontinuous on $L^1$, we get that $\omega_n\overset{\ast}{\rightharpoonup} \overline{\omega}$, for $n\rightarrow\infty$.
\endproof

\begin{proof}[Proof of Proposition~\ref{prop:Wr_convergence}] 
The thesis follows from Lemma~\ref{prop:lemma1} and from the uniform bound
\[\|\omega_r(t_n)\|_\infty\leq 2 \|\omega(0)\|_\infty,\]
for any $n\geq 1$.
This bound is a consequence of \cite[prop.\,5.4]{MoVi2022}, since
\[
\|\omega_r(t)\|_\infty=\|\omega(t)-\omega_s(t)\|_\infty\leq \|\omega(t)\|_\infty + \|\omega_s(t)\|_\infty \leq 2\|\omega_0\|_\infty.\]
\end{proof}

Hence, the states $\omega_s$, observed in our long-time numerical simulations, are natural candidates for the weak-$*$ vorticity limits.
We notice that minimal flows seem too restrictive as candidates for long-time asymptotics on $\Ss^2$; 
essentially, they allow only two vortex condensates, or the shape of similar signed condensates would have to balance perfectly, which is physically unlikely and numerically unseen.
Instead, it is common to see \emph{branching} in the functional dependence between vorticity and stream function \cite[fig.\,7 and 8]{MoVi2022}.

The canonical splitting offers a remedy: a selection criterion for a subset of mixing operators to define less restrictive minimal flows, namely 
\begin{align*}
 \mathcal{K}_\Pi=\lbrace\Pi_\psi \mid \Delta\psi\in \overline{\mathcal{O}(\omega_{0})}^*&\cap\lbrace H(\Delta\psi) =H(\omega_0)\rbrace\\
 &\cap\lbrace M(\Delta\psi) =M(\omega_0)\rbrace\rbrace,   
\end{align*}
where $M$ is the angular momentum.
Minimal flow with respect to these mixing operators might be compatible with the numerical observations (more than two blobs and branching), contrary to the full set of mixing operators, for which minimal flows cannot have branching and therefore generically necessitates two blobs only.    

\section{A conjecture on the long-time behavior and integrability}\label{sec:integrability-and-long-time}

One way to approach the asymptotic, long-time behavior is to ask what is contained in the $\omega$-limit set
\begin{equation}\label{eq:omega_limit_set}
    \Omega_+(\omega_0) = \bigcap_{s\geq 0} \overline{\{S_t(\omega_0)\mid t\geq s \}}^*   . 
\end{equation}
Shnirelman~\cite{Sh2013} conjectured that $\Omega_+(\omega_0)$ is compact in $L^2$. 
This conjecture can be understood as follows.
$\Omega_+(\omega_0)$ is indeed compact in the weak $L^2$ topology, so the $L^2$ norm (enstrophy) attains a minimum $\hat{\omega}$ on it.
We observe that $\overline{\{S_t(\hat{\omega})\mid t\geq 0 \}}^*\subset\Omega_+(\omega_0)$ and also the $L^2$ norm is constant on $\overline{\{S_t(\hat{\omega})\mid t\geq 0 \}}^*$.
Therefore, any weakly convergent sequence in $\overline{\{S_t(\hat{\omega})\mid t\geq 0 \}}^*$ is also strongly convergent in $L^2$.
Consequently, the conjecture of Shnirelman is that no further mixing (loss of enstrophy) can occur for initial data in $\Omega_+(\omega_0)$.

The conjecture does not imply that steady state is reached (e.g.,~a minimal flow).
Indeed, unsteadiness can persist indefinitely, without further mixing.
From batches of numerical experiments \cite{MoVi2020}, we conjecture that persistent unsteadiness in $\Omega_+(\omega_0)$ is connected to integrable configuration of ``vortex blob dynamics'', where the centers of mass of $n$ fixed-shape blobs interact (\emph{cf.}~Fig.\,\ref{fig:non-vanish-final_s}).
Depending on the domain's geometry, there is a threshold $n$ for which the dynamics is always integrable \cite{MoVi2020b}.
The underlying mechanism is then that vortex mixing continues until a large-scale almost integrable configuration of vortex blobs is reached.
Being quasi-periodic, it then acts as a barrier for further mixing, trapping the vorticity in persistent unsteadiness.
On $\Ss^2$, the thresholds are $n=3$ for generic initial data, and $n=4$ for vanishing angular momentum.
Long-time numerical simulations with Zeitlin's model align with this conjecture: for generic (smooth) initial data they typically settle on $n=3$ blobs (e.g.~Fig.\,\ref{fig:non-vanishing}), and for vanishing momentum on $n=4$ blobs; see \cite[Sec.\,4]{MoVi2020} for details.








\appendix
\begingroup
\allowdisplaybreaks

\section{Proof of Theorem~\ref{thm:convergence}}\label{appendix:convergence_proof}

Let $T_N$ denotes the projection of a function from $L^2(\Ss^2)$ to $\su(N)$.
In the following calculations, we mainly follow the approach of Gallagher~\cite{Ga2002} and the results by Flandoli, Pappalettera, and Viviani~\cite[Appendix\, A]{FlPaVi2022}.
The basic strategy is to show that the Zeitlin flow $W(t)\in\mathfrak{su}(N)$ lifted to $L^2$ via $\omega_N(t) \coloneqq T_N^*W(t)$ fulfills a differential inequality of the form
\begin{equation*}
    \frac{d}{dt}\lVert \omega_N(t) - \omega(t) \rVert \leq a_N + b_N \lVert \omega_N(t) - \omega(t) \rVert
\end{equation*}
such that $a_N \to 0$ and $b_N$ is bounded as $N\to \infty$ and $0 \leq t \leq t_{\it max}< \infty$.
Convergence then follow from the Grönwall lemma if $\lim_{N\to\infty}\lVert \omega_N(0) - \omega(0) \rVert = 0$.
To achieve the differential inequality we need to understand how the matrix commutator $[\cdot,\cdot]$ interacts with the Poisson bracket $\{\cdot,\cdot\}$.
This is done by studying the relation between the corresponding Lie algebra structure constants relative to the spherical harmonics basis (which is available for both functions and matrices).

We begin with the following result (with the convention that $\lVert\cdot\rVert_s$ denotes the $H^s(\Ss^2)$ Sobolev norm, $\lVert\cdot\rVert_0$ applied to matrices is the scaled Frobenius norm, and $[\cdot\,,\cdot]_N \coloneqq [\cdot\,,\cdot]/\hbar_N$).

\begin{lemma}[$L^2$-bracket convergence]\label{thm:La-conv}
Let $f\in H^{1+\epsilon}(\Ss^2)$ and $g\in H^{3+\varepsilon}(\Ss^2)$, we have that for any $\varepsilon>0$, $0<\alpha<2/3$, and $0<\beta<2/5$ there exist positive constants $C(\alpha)$, $C(\beta)$, $C(\alpha,\epsilon)$, and $C(\beta,\epsilon)$ such that
\begin{align*}
&\|[T_N\Delta^{-1}f,T_Ng]_N-T_N\lbrace\Delta^{-1}f,g\rbrace\|_0\leq \\
&\dfrac{C(\alpha)}{N^{2(1-3/2\alpha)}}\|f\|_{0}\|g\|_{1}
+\dfrac{C(\beta)}{N^{2(1-5/2\beta)}}\|f\|_{0}\|g\|_{0}+\\
&\min\left\lbrace \frac{C(\alpha,\varepsilon)}{N^{\alpha\varepsilon}}
\|f\|_{1+\varepsilon},C(\alpha) \textit{o}(1)\|f\|_0\right\rbrace\|g\|_1+\dfrac{C(\beta,\varepsilon)}{N^{\beta\varepsilon}}\|f\|_{0}\|g\|_{3+\varepsilon},
\end{align*}
where $o(1)\to 0$ as $N\to\infty$.
\end{lemma}
\proof
Let $C^{(N)l_3m_3}_{l_1m_1,l_2m_2}$ and $C^{l_3m_3 }_{l_1m_1,l_2m_2}$ denote the structure constants of $[\cdot,\cdot]_N$ and $\{\cdot,\cdot\}$ relative to the spherical harmonics basis (see Hoppe~\cite{Ho1982} and Yoshida~\cite{Yo1997}), and
let $N$ be an even integer.
We want to estimate the difference
\begin{align*}
&\|[T_N\Delta^{-1}f,T_Ng]_N-T_N\lbrace\Delta^{-1}f,g\rbrace\|_0^2 \lesssim \\
&\sum_{\substack{l_1<N,l_2<N\\ -l_1\leq m_1\leq l_1\\-l_2\leq m_2\leq l_2}}
\sum_{\substack{l_3=|l_1-l_2|+1\\ -l_3\leq m_3\leq l_3}}^{{\min\lbrace N,l_1+l_2-1\rbrace}}
|C^{(N)l_3m_3}_{l_1m_1,l_2m_2}-C^{l_3m_3 }_{l_1m_1,l_2m_2}|^2|\frac{f_{l_1m_1}}{l_1(l_1+1)}g_{l_2m_2}|^2=\\
&\sum_{\substack{l_1<N^\alpha\\l_2<N^\beta\\ -l_1\leq m_1\leq l_1\\ -l_2\leq m_2\leq l_2}}
\sum_{\substack{l_3=|l_1-l_2|+1\\ -l_3\leq m_3\leq l_3}}^{\min\lbrace N,l_1+l_2-1\rbrace}
|C^{(N)l_3m_3 }_{l_1m_1,l_2m_2}-C^{l_3m_3 }_{l_1m_1,l_2m_2}|^2|\frac{f_{l_1m_1}}{l_1(l_1+1)}g_{l_2m_2}|^2+\\
&\sum_{\substack{N^\alpha\leq l_1<N\\l_2<N^\beta\\ -l_1\leq m_1\leq l_1\\ -l_2\leq m_2\leq l_2}}
\sum_{\substack{l_3=|l_1-l_2|+1\\ -l_3\leq m_3\leq l_3}}^{{\min\lbrace N,l_1+l_2-1\rbrace}}
|C^{(N)l_3m_3 }_{l_1m_1,l_2m_2}-C^{l_3m_3 }_{l_1m_1,l_2m_2}|^2|\frac{f_{l_1m_1}}{l_1(l_1+1)}g_{l_2m_2}|^2+\\
&\sum_{\substack{l_1<N^\alpha\\N^\beta\leq l_2<N\\ -l_1\leq m_1\leq l_1\\ -l_2\leq m_2\leq l_2}}
\sum_{\substack{l_3=|l_1-l_2|+1\\ -l_3\leq m_3\leq l_3}}^{{\min\lbrace N,l_1+l_2-1\rbrace}}
|C^{(N)l_3m_3 }_{l_1m_1,l_2m_2}-C^{l_3m_3 }_{l_1m_1,l_2m_2}|^2|\frac{f_{l_1m_1}}{l_1(l_1+1)}g_{l_2m_2}|^2+\\
&\sum_{\substack{N^\alpha\leq l_1\leq N\\N^\beta\leq l_2<N\\ -l_1\leq m_1\leq l_1\\-l_2\leq m_2\leq l_2}}
\sum_{\substack{l_3=|l_1-l_2|+1\\ -l_3\leq m_3\leq l_3}}^{{\min\lbrace N,l_1+l_2-1\rbrace}}
|C^{(N)l_3m_3 }_{l_1m_1,l_2m_2}-C^{l_3m_3 }_{l_1m_1,l_2m_2}|^2|\frac{f_{l_1m_1}}{l_1(l_1+1)}g_{l_2m_2}|^2,
\end{align*}
where $0<\alpha,\beta<1$ have to be determined.
Notice that the sum \[\sum_{\substack{l_3=|l_1-l_2|+1\\ -l_3\leq m_3\leq l_3}}^{{\min\lbrace N,l_1+l_2-1\rbrace}}\] has $\min\lbrace N,2l_1-2,2l_2-2\rbrace$ different values of $l_3$, which we repeatedly use in the following steps.

\textbf{Step 1.} For the first term,  we have to divide sum in two cases: $l_2<l_1\land N^\beta$ and $l_1\leq l_2\land N^\alpha$.
Using \cite[prop.\,14\,(1)]{FlPaVi2022}, we get
\begin{align*}
&\sum_{\substack{l_1<N^\alpha\\l_2<l_1\land N^\beta\\ -l_1\leq m_1\leq l_1\\-l_2\leq m_2\leq l_2}}
\sum_{\substack{l_3=|l_1-l_2|+1\\ -l_3\leq m_3\leq l_3}}^{{\min\lbrace N,l_1+l_2-1\rbrace}}
|C^{(N)l_3m_3 }_{l_1m_1,l_2m_2}-C^{l_3m_3 }_{l_1m_1,l_2m_2}|^2|\frac{f_{l_1m_1}}{l_1(l_1+1)}g_{l_2m_2}|^2\\
&\leq \frac{C}{N^4}\sum_{\substack{l_1<N^\alpha\\ -l_1\leq m_1\leq l_1}}
\sum_{\substack{l_2<l_1\\ -l_2\leq m_2\leq l_2}}\sum_{\substack{l_3=|l_1-l_2|+1\\ -l_3\leq m_3\leq l_3}}^{{\min\lbrace N,l_1+l_2-1\rbrace}}
l_1^{8}|\frac{f_{l_1m_1}}{l_1(l_1+1)}g_{l_2m_2}|^2\\
&\leq \frac{C}{N^4}\sum_{\substack{l_1<N^\alpha\\ -l_1\leq m_1\leq l_1}}
\sum_{\substack{l_2<l_1\\ -l_2\leq m_2\leq l_2}}
l_1^4l_2^2|f_{l_1m_1}g_{l_2m_2}|^2\\
&\leq \frac{C}{N^4}\left(\sum_{\substack{l_1<N^\alpha\\ -l_1\leq m_1\leq l_1}}
l_1^{2}|f_{l_1m_1}|\right)^2\|g\|^2_1
\\
&\leq \frac{C}{N^4}\sum_{l_1<N^\alpha}
l_1^{2\eta+1}\|f\|^2_{2-\eta}\|g\|^2_1
\\
&\leq C N^{2\alpha(\eta+1)-4}\|f\|^2_{2-\eta}\|g\|^2_1,
\end{align*}
Hence, choosing $\eta=2$, we get that we can take $\alpha<2/3$ and the best bound is:
\[
CN^{4(3/2\alpha-1)}\|f\|^2_0\|g\|^2_1.
\]
For the case $l_1\leq l_2\land N^\alpha$, we proceed analogously:
\begin{align*}
&\frac{C}{N^4}\sum_{\substack{l_2<N^\beta\\l_1\leq l_2\land N^\alpha\\ -l_1\leq m_1\leq l_1\\ -l_2\leq m_2\leq l_2}}
\sum_{\substack{l_3=|l_1-l_2|+1\\ -l_3\leq m_3\leq l_3}}^{{\min\lbrace N,l_1+l_2-1\rbrace}}
l_2^{8}|\frac{f_{l_1m_1}}{l_1(l_1+1)}g_{l_2m_2}|^2\\
&\leq \frac{C}{N^4}\sum_{\substack{l_2<N^\beta\\ -l_2\leq m_2\leq l_2}}
\sum_{\substack{l_1\leq l_2\\ -l_1\leq m_1\leq l_1}}
l_2^{8}l_1^2|\frac{f_{l_1m_1}}{l_1(l_1+1)}g_{l_2m_2}|^2\\
&\leq \frac{C}{N^4}\left(\sum_{\substack{l_2<N^\beta\\ -l_2\leq m_2\leq l_2}}
l_2^{4}|g_{l_2m_2}|\right)^2\|f\|^2_{-1}
\\
&\leq \frac{C}{N^4}\sum_{l_2<N^\beta}
l_2^{2\eta+1}\|g\|^2_{4-\eta}\|f\|^2_{-1}
\\
&\leq C N^{2\beta(\eta+1)-4}\|g\|_{4-\eta}^2\|f\|_{-1}^2.
\end{align*}
Hence, choosing $\eta=4$, we get that we can take $\beta<2/5$ and the bound is:
\[
CN^{4(3\beta-1)}\|f\|_{-1}^2\|g\|^2_0.
\]
Therefore, we have that the first term is bounded by:
\begin{align*}
&\sum_{\substack{l_1<N^\alpha\\l_2<N^\beta\\ -l_1\leq m_1\leq l_1\\ -l_2\leq m_2\leq l_2}}
\sum_{\substack{l_3=\\|l_1-l_2|+1\\ -l_3\leq m_3\leq l_3}}^{\substack{\min\lbrace N,\\l_1+l_2-1\rbrace}}
|C^{(N)l_3m_3 }_{l_1m_1,l_2m_2}-C^{l_3m_3 }_{l_1m_1,l_2m_2}|^2|\frac{f_{l_1m_1}}{l_1(l_1+1)}g_{l_2m_2}|^2\\
&\leq 
\dfrac{C(\alpha)}{N^{4(1-3/2\alpha)}}\|f\|^2_{0}\|g\|^2_{1}+\dfrac{C(\beta)}{N^{4(1-5/2\beta)}}\|f\|^2_{-1}\|g\|^2_{0},
\end{align*}
for any $\alpha<2/3$, $\beta<2/5$.

\textbf{Step 2.} For the second term, we proceed analogously, using the result in \cite[prop.\,14\,(2)]{FlPaVi2022}.
Since, $\beta<\alpha$, we always have $l_2<l_1$. Therefore,
\begin{align*}
&\sum_{\substack{N^\alpha<l_1<N\\l_2<N^\beta\\ -l_1\leq m_1\leq l_1\\ -l_2\leq m_2\leq l_2}}
\sum_{\substack{l_3=\\|l_1-l_2|+1\\ -l_3\leq m_3\leq l_3}}^{\substack{\min\lbrace N,\\l_1+l_2-1\rbrace}}
|C^{(N)l_3m_3 }_{l_1m_1,l_2m_2}-C^{l_3m_3 }_{l_1m_1,l_2m_2}|^2|\frac{f_{l_1m_1}}{l_1(l_1+1)}g_{l_2m_2}|^2\\	
&\leq {C}\sum_{\substack{N^\alpha<l_1<N\\ -l_1\leq m_1\leq l_1}}
\sum_{\substack{l_2<N^\beta\\ -l_2\leq m_2\leq l_2}}\sum_{\substack{l_3=|l_1-l_2|+1\\ -l_3\leq m_3\leq l_3}}^{{\min\lbrace N,l_1+l_2-1\rbrace}}
l_1^{4}|\frac{f_{l_1m_1}}{l_1(l_1+1)}g_{l_2m_2}|^2\\
&\leq {C}\sum_{\substack{N^\alpha<l_1<N\\ -l_1\leq m_1\leq l_1}}
\sum_{\substack{l_2<N^\beta\\ -l_2\leq m_2\leq l_2}}
l_2^2|f_{l_1m_1}g_{l_2m_2}|^2\\
&\leq {C}\sum_{\substack{N^\alpha<l_1<N\\ -l_1\leq m_1\leq l_1}}
|f_{l_1m_1}|^2\|g\|_1^2\\
&\leq {C}\min\left\lbrace \sum_{\substack{N^\alpha<l_1<N\\ -l_1\leq m_1\leq l_1}}l_1^{-2-2\varepsilon}
\|f\|_{1+\varepsilon}^2 , \textit{o}(1)\|f\|_0^2\right\rbrace\|g\|_1^2\\
&\leq {C}\min\left\lbrace N^{-2\alpha\varepsilon}
\|f\|_{1+\varepsilon}^2, \textit{o}(1)\|f\|_0^2\right\rbrace\|g\|_1^2,
\end{align*}
for some $\varepsilon>0$ and $\textit{o}(1)\rightarrow 0$, for $N\rightarrow\infty$.
Therefore, we have that the second term is bounded by:
\begin{align*}
&\sum_{\substack{N^\alpha<l_1<N\\l_2<N^\beta\\ -l_1\leq m_1\leq l_1\\ -l_2\leq m_2\leq l_2}}
\sum_{\substack{l_3=\\|l_1-l_2|+1\\ -l_3\leq m_3\leq l_3}}^{\substack{\min\lbrace N,\\l_1+l_2-1\rbrace}}
|C^{(N)l_3m_3 }_{l_1m_1,l_2m_2}-C^{l_3m_3 }_{l_1m_1,l_2m_2}|^2|\frac{f_{l_1m_1}}{l_1(l_1+1)}g_{l_2m_2}|^2\\
&\leq 
\min\left\lbrace \frac{C(\alpha,\varepsilon)}{N^{2\alpha\varepsilon}}
\|f\|_{1+\varepsilon}^2,C(\alpha) \textit{o}(1)\|f\|_0^2\right\rbrace\|g\|_1^2,
\end{align*}
for any $\alpha<2/3$, $\beta<2/5$ and some $\varepsilon>0$.

\textbf{Step 3.} For the third term, we proceed by dividing the cases for which $l_2<l_1$ and $l_1\leq l_2 \land N^\alpha$. 
When $l_2<l_1$, we can repeat the same calculations of Step 1, getting the same bound.
When $l_1\leq l_2 \land N^\alpha$, we consider
\begin{align*}
&\sum_{\substack{N^\beta<l_2<N\\l_1\leq l_2\land N^\alpha\\ -l_2\leq m_2\leq l_2\\-l_1\leq m_1\leq l_1}}
\sum_{\substack{l_3=|l_1-l_2|+1\\ -l_3\leq m_3\leq l_3}}^{{\min\lbrace N,l_1+l_2-1\rbrace}}
|C^{(N)l_3m_3 }_{l_1m_1,l_2m_2}-C^{l_3m_3 }_{l_1m_1,l_2m_2}|^2|\frac{f_{l_1m_1}}{l_1(l_1+1)}g_{l_2m_2}|^2\\	
&\leq {C}\sum_{\substack{N^\beta<l_2<N\\ -l_2\leq m_2\leq l_2}}
\sum_{\substack{l_1<l_2\\ -l_1\leq m_1\leq l_1}}\sum_{\substack{l_3=|l_1-l_2|+1\\ -l_3\leq m_3\leq l_3}}^{{\min\lbrace N,l_1+l_2-1\rbrace}}
l_2^{4}|\frac{f_{l_1m_1}}{l_1(l_1+1)}g_{l_2m_2}|^2\\
&\leq {C}\sum_{\substack{N^\beta<l_2<N\\ -l_2\leq m_2\leq l_2}}
\sum_{\substack{l_1<l_2\\ -l_1\leq m_1\leq l_1}}
l_2^{4}l_1^2|\frac{f_{l_1m_1}}{l_1(l_1+1)}g_{l_2m_2}|^2\\
&\leq C \sum_{\substack{N^\beta<l_2<N\\ -l_2\leq m_2\leq l_2}}
l_2^{4}|g_{l_2m_2}|^2\|f\|^2_{-1}
\\
&\leq C \left(\sum_{\substack{N^\beta<l_2<N\\ -l_2\leq m_2\leq l_2}}
l_2^{2}|g_{l_2m_2}|\right)^2\|f\|^2_{-1}
\\
&\leq C \sum_{\substack{N^\beta<l_2<N\\ -l_2\leq m_2\leq l_2}}
l_2^{-2-2\varepsilon}\|f\|^2_{-1}\|g\|^2_{3+\varepsilon}
\\
&\leq C N^{-2\beta\varepsilon}\|f\|^2_{-1}\|g\|^2_{3+\varepsilon}.
\end{align*}
Therefore, the third term admits the bound
\begin{align*}
&\sum_{\substack{l_1<N^\alpha\\ N^\beta<l_2<N\\-l_2\leq m_2\leq l_2\\-l_1\leq m_1\leq l_1}}
\sum_{\substack{l_3=|l_1-l_2|+1\\ -l_3\leq m_3\leq l_3}}^{{\min\lbrace N,l_1+l_2-1\rbrace}}
|C^{(N)l_3m_3 }_{l_1m_1,l_2m_2}-C^{l_3m_3 }_{l_1m_1,l_2m_2}|^2|\frac{f_{l_1m_1}}{l_1(l_1+1)}g_{l_2m_2}|^2\\
&\leq 
\dfrac{C(\alpha)}{N^{4(1-3/2\alpha)}}\|f\|^2_{0}\|g\|^2_{1}+\dfrac{C(\beta,\varepsilon)}{N^{2\beta\varepsilon}}\|f\|^2_{-1}\|g\|^2_{3+\varepsilon},
\end{align*}
for $\alpha<2/3$, $\beta<2/5$, and $\varepsilon>0$.

\textbf{Step 4.} For the fourth term, we proceed by dividing the cases for which $l_2<l_1$ and $l_1\leq l_2$. 
In both cases, we get the same bounds as obtained in Step 2 and 3.
Therefore, the fourth term admits the bound
\begin{align*}
&\sum_{\substack{N^\alpha\leq l_1\leq N\\N^\beta\leq l_2<N \\ -l_1\leq m_1\leq l_1\\-l_2\leq m_2\leq l_2}}
\sum_{\substack{l_3=|l_1-l_2|+1\\ -l_3\leq m_3\leq l_3}}^{{\min\lbrace N,l_1+l_2-1\rbrace}}
|C^{(N)l_3m_3 }_{l_1m_1,l_2m_2}-C^{l_3m_3 }_{l_1m_1,l_2m_2}|^2|\frac{f_{l_1m_1}}{l_1(l_1+1)}g_{l_2m_2}|^2\\
&\leq \min\left\lbrace \frac{C(\alpha,\varepsilon)}{N^{2\alpha\varepsilon}}
\|f\|_{1+\varepsilon}^2,C(\alpha) \textit{o}(1)\|f\|_0^2\right\rbrace\|g\|_1^2+\\
&\dfrac{C(\alpha)}{N^{4(1-3/2\alpha)}}\|f\|^2_{0}\|g\|^2_{1}+\dfrac{C(\beta,\varepsilon)}{N^{2\beta\varepsilon}}\|f\|^2_{-1}\|g\|^2_{3+\varepsilon},
\end{align*}
for $\alpha<2/3$, $\beta<2/5$, and $\varepsilon>0$.
\endproof

We now obtain $L^2$ convergence of the quantized bracket as $N\to\infty$.
\begin{corollary}[Spatial convergence]\label{thm:consistency}
In the framework of Lemma~\ref{thm:La-conv}, we have that:
\begin{align*}
&\|T_N^*[T_N\Delta^{-1}f,T_Ng]_N-\lbrace\Delta^{-1}f,g\rbrace\|_0\leq  \\
&\dfrac{C(\alpha)}{N^{2(1-3/2\alpha)}}\|f\|_{0}\|g\|_{1}
+\dfrac{C(\beta)}{N^{2(1-5/2\beta)}}\|f\|_{0}\|g\|_{0}+\\
&\min\left\lbrace \frac{C(\alpha,\varepsilon)}{N^{\alpha\varepsilon}}
\|f\|_{1+\varepsilon},C(\alpha) \textit{o}(1)\|f\|_0\right\rbrace\|g\|_1+\dfrac{C(\beta,\varepsilon)}{N^{\beta\varepsilon}}\|f\|_{0}\|g\|_{3+\varepsilon},
\end{align*}
for some $\varepsilon>0$, $0<\alpha<2/3$ and $0<\beta<2/5$.
\end{corollary}
\proof
We have that:
\begin{align*}
 & \|T_N^*[T_N\Delta^{-1}f,T_Ng]_N-\lbrace\Delta^{-1}f,g\rbrace\|_0\\
 &=\|[T_N\Delta^{-1}f,T_Ng]_N-T_N\lbrace\Delta^{-1}f,g\rbrace\|_0+\|(Id-T_N^*T_N)\lbrace\Delta^{-1}f,g\rbrace\|_0.  
\end{align*}
The first term has been estimate in Lemma~\ref{thm:La-conv}, so we have to study only the residual term 
\begin{align*}
&\|(Id-T_N^*T_N)\lbrace\Delta^{-1}f,g\rbrace\|_0^2=\\
&\sum_{\substack{l_1>N\\ -l_1\leq m_1\leq l_1}}
\sum_{\substack{l_2>N\\ -l_2\leq m_2\leq l_2}}\sum_{\substack{l_3=|l_1-l_2|+1\\ -l_3\leq m_3\leq l_3}}^{l_1+l_2-1}
|C^{l_3m_3 }_{l_1m_1,l_2m_2}|^2|\frac{f_{l_1m_1}}{l_1(l_1+1)}g_{l_2m_2}|^2.
\end{align*}
Repeating the same calculations of Lemma~\ref{thm:La-conv}, dividing in the two cases $l_1<l_2$ and $l_2\leq l_1$, and using the result in \cite[Lemma 12]{FlPaVi2022}, we get that:
\begin{align*}
&\sum_{\substack{l_1>N\\ -l_1\leq m_1\leq l_1}}
\sum_{\substack{l_2>N\\ -l_2\leq m_2\leq l_2}}\sum_{\substack{l_3=|l_1-l_2|+1\\ -l_3\leq m_3\leq l_3}}^{l_1+l_2-1}
|C^{l_3m_3 }_{l_1m_1,l_2m_2}|^2|\frac{f_{l_1m_1}}{l_1(l_1+1)}g_{l_2m_2}|^2\\
&=\sum_{\substack{l_1>N\\ -l_1\leq m_1\leq l_1}}
\sum_{\substack{l_2>l_1\\ -l_2\leq m_2\leq l_2}}\sum_{\substack{l_3=|l_1-l_2|+1\\ -l_3\leq m_3\leq l_3}}^{l_1+l_2-1}
|C^{l_3m_3 }_{l_1m_1,l_2m_2}|^2|\frac{f_{l_1m_1}}{l_1(l_1+1)}g_{l_2m_2}|^2+\\
&\sum_{\substack{l_1>N\\ -l_1\leq m_1\leq l_1}}
\sum_{\substack{N<l_2\leq l_1\\ -l_2\leq m_2\leq l_2}}\sum_{\substack{l_3=|l_1-l_2|+1\\ -l_3\leq m_3\leq l_3}}^{l_1+l_2-1}
|C^{l_3m_3 }_{l_1m_1,l_2m_2}|^2|\frac{f_{l_1m_1}}{l_1(l_1+1)}g_{l_2m_2}|^2\\
&\leq\sum_{\substack{l_1>N\\ -l_1\leq m_1\leq l_1}}
\sum_{\substack{l_2>l_1\\ -l_2\leq m_2\leq l_2}}l_1^2 l_2^4 |\frac{f_{l_1m_1}}{l_1(l_1+1)}g_{l_2m_2}|^2+\\
&\sum_{\substack{l_1>N\\ -l_1\leq m_1\leq l_1}}
\sum_{\substack{N<l_2\leq l_1\\ -l_2\leq m_2\leq l_2}}l_1^4 l_2^2|\frac{f_{l_1m_1}}{l_1(l_1+1)}g_{l_2m_2}|^2\\
&\leq \dfrac {C(\gamma)}{N^\gamma}\lbrace \|f\|^2_{-1}\|g\|_{3+\gamma}^2+\|f\|^2_{0}\|g\|_{2+\gamma}^2\rbrace\\
&\leq \dfrac {C(\gamma)}{N^\gamma}\|f\|^2_{0}\|g\|_{3+\gamma}^2,
\end{align*}
for some $\gamma>0$.
Hence, the residual term $\|(Id-T_N^*T_N)\lbrace\Delta^{-1}f,g\rbrace\|_0$ does not worsen the bound of Lemma~\ref{thm:La-conv}, for $\gamma=\beta\varepsilon$.
\endproof

The last ingredient for the proof of the convergence is the following:
\begin{proposition}\label{prop:bracket_SH}
$\displaystyle
\|T_N\lbrace\Delta^{-1}f,g\rbrace\|_0
\leq C\|f\|_{-1}\|\nabla g\|_{L^\infty}.
$
\end{proposition}
\proof
\begin{align*}
\|T_N\lbrace\Delta^{-1}f,g\rbrace\|_0^2
&\leq \|\lbrace\Delta^{-1}f,g\rbrace\|_0^2\\
&=\int |\nabla\Delta^{-1}f\times\nabla g|^2\\\
&\leq \int |\nabla\Delta^{-1}f|^2\|\nabla g\|_{L^\infty}^2\\
&\leq C\|f\|_{-1}^2\|\nabla g\|_{L^\infty}^2.
\end{align*}
\endproof
\begin{theorem}[Space-time convergence]\label{thm:space-time-convergence-detailed}
Let $\omega=\omega(t,x)$ and $W=W(t)$ be the solutions to the Euler equations \eqref{eq:vorteq2} and Euler--Zeitlin equations \eqref{eq:zeitlin_model}, with initial data $\omega(0)\in H^{3+\varepsilon}(\mathbb{S}^2)$ for some $\varepsilon>0$ and $W(0)\in\su(N)$. 
Let $\omega_N(t):=T_N^*W(t)$.
Then, for any $t\geq 0$, and any $\varepsilon>0$, $0<\alpha<2/3$ and $0<\beta<2/5$, there exist positive constants $C(\alpha), C(\beta), C(\alpha,\varepsilon), C(\beta,\varepsilon)$ such that
\begin{align*}
 &\| \omega(t) - \omega_N(t) \|_0\leq  \Bigg\lbrace\| \omega(0) - \omega_N(0) \|_0+\\
 &\|\omega(0)\|_{0}\int_0^t\dfrac{C(\alpha)}{N^{2(1-3/2\alpha)}}\|\omega(s)\|_{1} ds+\dfrac{t \, C(\beta)}{N^{2(1-5/2\beta)}}\|\omega(0)\|_{0}^2+\\
&\dfrac{C(\alpha,\varepsilon)}{N^{\alpha\varepsilon}}\int_0^t\|\omega(s)\|_{1+\varepsilon}\|\omega(s)\|_{1}ds+\dfrac{C(\beta,\varepsilon)}{N^{\beta\varepsilon}}\|\omega(0)\|_{0}\int_0^t\|\omega(s)\|_{3+\varepsilon}ds\Bigg\rbrace\cdot\\
&\exp\Bigg\lbrace \int_0^t\dfrac{C(\alpha)}{N^{2(1-3/2\alpha)}}\|\omega(s)\|_{1}ds+\dfrac{t\, C(\beta)}{N^{2(1-5/2\beta)}}\|\omega(0)\|_{0}+\\
&\int_0^tC(\alpha)\|\omega(s)\|_{1} +\dfrac{C(\beta,\varepsilon)}{N^{\beta\varepsilon}}\|\omega(s)\|_{3+\varepsilon}    +\|\nabla\omega(s)\|_{L^\infty}ds\Bigg\rbrace .
\end{align*}
\end{theorem}

\proof Following \cite{Ga2002}, we consider:
\begin{equation}\label{eq:time_derivative}
\begin{array}{ll}
     & \dfrac{1}{2}\dfrac{d}{dt}\|\omega_N(t)-\omega(t)\|_0^2= \\
     & = \langle \omega_N(t)-\omega(t), T_N^*[\Delta^{-1}_NW(t),W(t)]_N-\lbrace\Delta^{-1}\omega(t),\omega(t)\rbrace\rangle.
\end{array}
\end{equation}
From the bi-invariance of the (scaled) Frobenius inner product on $\mathfrak{u}(N)$, we get
\begin{align*}
0 =&\langle W(t)-T_N\omega(t), [\Delta^{-1}_NW(t),W(t)-T_N\omega(t)]_N\rangle \\
=&\langle T_N(\omega_N(t)-\omega(t)), [\Delta^{-1}_NW(t),W(t)-T_N\omega(t)]_N\rangle \\
=&\langle \omega_N(t)-\omega(t), T_N^*[\Delta^{-1}_NW(t),W(t)-T_N\omega(t)]_N\rangle. 
\end{align*}
Therefore, equation \eqref{eq:time_derivative} can be written as:
\begin{align*}
&\dfrac{1}{2}\dfrac{d}{dt}\|\omega_N(t)-\omega(t)\|_{0}^2= \\
&= \langle \omega_N(t)-\omega(t), T_N^*[\Delta^{-1}_NW(t),T_N\omega(t)]_N-\lbrace\Delta^{-1}\omega(t),\omega(t)\rbrace\rangle
\\
&=\langle \omega_N(t)-\omega(t), T_N^*[\Delta^{-1}_N T_N\omega(t),T_N\omega(t)]_N-\lbrace\Delta^{-1}\omega(t),\omega(t)\rbrace\rangle\\
&+\langle \omega_N(t)-\omega(t), T_N^*[\Delta^{-1}_N(W(t)-T_N\omega(t)),T_N\omega(t)]_N\rangle.
\end{align*}
The second term of the right-hand side can be written as
\begin{align*}
&\langle \omega_N(t)-\omega(t), T_N^*[\Delta^{-1}_N(W(t)-T_N\omega(t)),T_N\omega(t)]_N\rangle=\\
&=\langle \omega_N(t)-\omega(t), T_N^*[\Delta^{-1}_N(W(t)-T_N\omega(t)),T_N\omega(t)]_N\\
&-T_N^*T_N\lbrace\Delta^{-1}(\omega_N(t)-T_N^*T_N\omega(t)),T_N^*T_N\omega(t)\rbrace\rangle\\
&+\langle \omega_N(t)-\omega(t),T_N^*T_N\lbrace\Delta^{-1}(\omega_N(t)-T_N^*T_N\omega(t)),T_N^*T_N\omega(t)\rbrace\rangle
\end{align*}
From Lemma~\ref{thm:La-conv}, Corollary~\ref{thm:consistency} and Proposition~\ref{prop:bracket_SH}, we can then deduce that: 
\begin{align*}
&\dfrac{1}{2}\dfrac{d}{dt}\|\omega_N(t)-\omega(t)\|_{0}^2 \leq \|\omega_N(t)-\omega(t)\|_{0}\;\cdot \\
&\Big(\dfrac{C(\alpha)}{N^{2(1-3/2\alpha)}}\|\omega(t)\|_{0}\|\omega(t)\|_{1}+\dfrac{C(\beta)}{N^{2(1-5/2\beta)}}\|\omega(t)\|_{0}\|\omega(t)\|_{0} \\
&+\dfrac{C(\alpha,\varepsilon)}{N^{\alpha\varepsilon}}\|\omega(t)\|_{1+\varepsilon}\|\omega(t)\|_{1} + \dfrac{C(\beta,\varepsilon)}{N^{\beta\varepsilon}}\|\omega(t)\|_{0}\|\omega(t)\|_{3+\varepsilon}\Big) + \\
& \|\omega_N(t)-\omega(t)\|^2_0\;\cdot\left(\dfrac{C(\alpha)}{N^{2(1-3/2\alpha)}}\|\omega(t)\|_{1}+\right.\\
&\left.\dfrac{C(\beta)}{N^{2(1-5/2\beta)}}\|\omega(t)\|_{0}+{C(\alpha)}\|\omega(t)\|_{1}+
\dfrac{C(\beta,\varepsilon)}{N^{\beta\varepsilon}}\|\omega(t)\|_{3+\varepsilon}+\|\nabla\omega\|_{L^\infty}\right).
\end{align*}
The thesis then follows from the Gr\"onwall lemma.
\endproof

\endgroup

\section{Relation to Berezin-Toeplitz quantization}\label{appendix:line_bundles}

Since Dirac~\cite{Di1925} outlined quantization theory it has ramified in different mathematical directions.
In this paper, we obtain quantization via unitary representation theory, which perhaps is the most direct route when applicable.
Non-commutative geometry (\emph{cf.}~\cite{Co1994}) offers a broader viewpoint, and geometric quantization (\emph{cf.}~\cite{So1966}) yet another.
The various frameworks are in general non-equivalent and applicable under different conditions.
But for the sphere they coincide; it is somehow the archetype of quantization (in non-commutative geometry the ``fuzzy sphere'', and in geometric quantization the tautological complex line bundle over the Riemann sphere $\mathbb{C}P^1$, as we shall see).

Theorem~\ref{thm:quantization_limit} above was stated and proved by Charles and Polterovich~\cite{ChPo2018} in the framework of \emph{Berezin-Toeplitz quantization} (which is a type of geometric quantization, \emph{cf.}~\cite{Be1975,Le2018}).
The purpose of this section is to show how quantization of the sphere via representations, as in section~\ref{sec:quant-via-repr} above, corresponds to Berezin-Toeplitz quantization of the Riemann sphere.
This result is essentially contained in the original work of Berezin~\cite{Be1975}, and fully explained in the monograph by Le~Floch~\cite{Le2018}, to which we refer for details.

Let $M$ be a compact Kähler manifold.
The objective is to construct a mapping (the Berezin-Toeplitz quantization) between the space of smooth functions $C^\infty(M,\mathbb{C})$ and a space of operators on a finite dimensional quantum state space (i.e., the Hilbert space of ``wave-functions'').
State space is constructed as holomorphic sections of a \new{tensor} power of a Hermitian line bundle over $M$.
The construction may appear abstruse at first, but is straightforward in the simplest setting $M=\mathbb{C}P^1$.
As we shall see, this corresponds to quantization on $\Ss^2$ as presented in section~\ref{sec:quant-via-repr}.

\begin{figure}
    \includegraphics{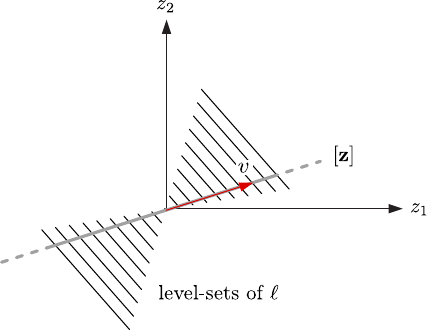}
    \caption{
    Illustration of the tautological line bundle $\mathcal{O}(-1)$.
    An element $([\mathbf z],v)\in \mathcal{O}(-1)$ consist of a one-dimensional complex linear subspace $[\mathbf z]\in \mathbb{C}P^1$ together with a vector in that subspace $v\in [\mathbf z]$.
    Thus, a linear form $\ell \in (\Cc^2)^*$ naturally pairs with elements of $\mathcal{O}(-1)$.}
    \label{fig:tautological}
\end{figure}

Recall that the Riemann sphere $\mathbb{C}P^1$ is the manifold of one-dimensional complex subspaces of $\mathbb{C}^2$.
It can be thought of as the quotient $(\mathbb{C}^2\backslash \{0 \})/\mathbb{C}_*\simeq \Ss^3/\Ss^1 \simeq \Ss^2$ equipped with the complex structure inherited from $\mathbb{C}^2$.
Since $\mathbb{C}^2$ admits a large space of holomorphic functions (e.g., any polynomial in $\mathbf z=(z_1,z_2)$), it is natural to suggest as state space the holomorphic functions descending to $\mathbb{C}P^1$.
However, the only such functions are the constants (since non-constant polynomials $p(z_1,z_2)$ cannot be invariant under the action of $\mathbb{C}_*$).
The remedy is to construct a complex vector bundle over $\mathbb{C}P^1$ and consider its holomorphic sections.

There is a natural vector bundle associated with the Riemann sphere: the fiber above $[\mathbf{z}]\in \mathbb{C}P^1$ is a vector $v\in [\mathbf{z}]$ of the co-set.
This is the \emph{tautological line bundle}, denoted $\mathcal{O}(-1)$ and illustrated in Fig.~\ref{fig:tautological}.
Since it is a sub-bundle of the trivial vector bundle $\mathbb{C}P^1\times \mathbb{C}^2$, it inherits the Hermitian inner product of~$\mathbb{C}^2$.
Furthermore, it is a holomorphic line bundle in the sense that it is a \new{holomorphic} manifold and the projection $\mathcal{O}(-1)\to\mathbb{C}P^1$ is a holomorphic mapping.


Consider now an element $p\in (\mathbb{C}^2)^*$, i.e., a linear function $p\colon \mathbb{C}^2\to \mathbb{C}$.
As illustrated in Fig.~\ref{fig:tautological}, there is a natural bundle pairing between $p$ and a smooth section $s \in C^\infty(\mathbb{C}P^1, \mathcal{O}(-1))$, namely
\begin{equation*}
    s([\mathbf{z}]) \mapsto p(s([\mathbf{z}])).
\end{equation*}
Thus, we can think of $p$ as a section of the dual bundle $\mathcal{O}(1) \equiv \mathcal{O}(-1)^*$.
Likewise, any section $s^*$ of $\mathcal{O}(1)$ gives rise to a function on $\Cc^2\backslash \{ 0\}$ given by $p(\mathbf{z}) = \langle s^*([\mathbf{z}]), \mathbf{z}\rangle_{[\mathbf{z}]}$.
The dual bundle $\mathcal{O}(1)$ is also a holomorphic line bundle, and if $s^*$ is a holomorphic section then $p(z_1,z_2) = a z_1 + b z_2$ for $a,b\in\mathbb{C}$.
(One can verify this by first observing that if $s^*$ is holomorphic, then $p(\mathbf{z})$ is holomorphic. Since $p$ is also 1-homogenous it must be a linear holomorphic function.)
On the other hand, the space of all smooth sections $C^\infty(\Cc P^1, \mathcal{O}(1))$ is, of course, infinite-dimensional and can act as the quantum state space via the Kostant-Souriau operator, which associates a smooth function $f\in C^\infty(\Cc P^1,\Rr)$ with the operator $\mathrm{KS}_1(f)\colon C^\infty(\Cc P^1, \mathcal{O}(1))\to C^\infty(\Cc P^1, \mathcal{O}(1))$ defined by
\begin{equation*}
    \mathrm{KS}_1(f) = f - i \nabla_{X_f}
\end{equation*}
where $f$ acts by multiplication, $\nabla$ is the connection associated with the Hermitian structure of the line bundle $\mathcal{O}(1)$ (i.e., the Chern connection), and $X_f$ is the Hamiltonian vector field for $f$.
The Kostant-Souriau operators fulfill the quantization condition\footnote{The Hermitian structure is scaled so the associated symplectic structure $\Omega$ fulfills $\int_{\Cc P^1}\Omega = 2\pi$.}
\begin{equation*}
    [\mathrm{KS}_1(f), \mathrm{KS}_1(g)] = -\new{i} \mathrm{KS}_1(\{f,g \}).
\end{equation*}

But the space $C^\infty\big(\Cc P^1, \mathcal{O}(1)\big)$ is inadequate as quantum state space: it is infinite-dimensional and admits far too many self-adjoint operators.
In contrast, the space of holomorphic sections of $\mathcal{O}(1)$ is too small (as it is given by $(\Cc^2)^*$).
We need something in-between.
The solution is to consider \new{the tensor product $\mathcal{O}(-k) \equiv \mathcal{O}(-1)^{\otimes k}$ and then take holomorphic sections of its dual bundle $\mathcal{O}(k) \equiv \mathcal{O}(-k)^*$.
Notice that $\mathcal{O}(k)$ is still a line bundle (since the fiber of $\mathcal{O}(-1)$ is one-dimensional), so this construction might appear fruitless.
But the point is that it admits more holomorphic sections.
Indeed, every homogenous polynomial $p\in \Cc_k[z_1,z_2]$ of degree $k$ gives rise to a $k$-linear form on $\mathcal{O}(-1)$ via
}
\begin{equation*}
    [\mathbf{z}]^k \ni (\lambda_1 \mathbf{z},\ldots,\lambda_k \mathbf{z}) \mapsto \lambda_1 \cdots \lambda_k p(\mathbf{z}).
\end{equation*}
Due to the homogeneity of $p$, this definition is independent of the representative $\mathbf{z} \in [\mathbf{z}]$.
\new{It defines a holomorphic section of $\mathcal{O}(k)$ (using that $k$-linear forms on $\mathcal{O}(-1)$ are naturally identified with sections of $(\mathcal{O}(-1)^{\otimes k})^* = \mathcal{O}(k)$).}
\new{Moreover, all holomorphic sections of $\mathcal{O}(k)$ are of this form, so} the space of homogenous polynomials $\Cc_k[z_1,z_2]$ is isomorphic to the space of holomorphic sections of $\mathcal{O}(k)$.

With this identification, consider the $L^2$ projection 
\begin{equation*}
\Pi\colon L^2(\Cc P^1, \mathcal{O}(k))\to \Cc_k[z_1,z_2]\qquad \text{(the Szegő projector.)}
\end{equation*}
The Berezin-Toeplitz operator for $f\in C^\infty(\Cc P^1,\Cc)$ is
\begin{equation*}
    \mathrm{BT}_{k}(f)\colon \Cc_{k}[z_1,z_2] \to \Cc_{k}[z_1, z_2], \quad \mathrm{BT}_{k}(f) = \Pi\circ f,
\end{equation*}
where $f$ acts by multiplication.
If $f$ is real-valued, then $\mathrm{BT}_k(f)$ is self-adjoint.

\new{Let $\mathrm{KS}_k(f)$ denote the generalization of $\mathrm{KS}_1(f)$ to the operator
\begin{equation*}
    \mathrm{KS}_k(f) = f - \frac{i}{k}\nabla^k_{X_f}\colon  C^\infty(\Cc P^1, \mathcal{O}(k)) \to C^\infty(\Cc P^1, \mathcal{O}(k)),
\end{equation*}
where $\nabla^k$ is the connection on $\mathcal{O}(k)$.}
Notice that the operator $\mathrm{BT}_{k}(f)$ is different from ``standard'' geometric quantization, which is $\Pi\circ \mathrm{KS}_k$ as an operator $\Cc_{k}[z_1,z_2]\to \Cc_{k}[z_1,z_2]$. 
Via Tuynman's~\cite{Tu1987} formula \new{$\Pi \circ 2 i \nabla^k_{X_f} \circ \Pi = \Pi \circ (\Delta f)\circ \Pi$} one concludes the relation $\Pi\circ\mathrm{KS}_k(f) = \mathrm{BT}_k(f) - \frac{1}{2k} \mathrm{BT}_k(\Delta f)$.

To obtain the connection to quantization via representation as described above, consider the Berezin-Toeplitz operators associated with the coordinates functions $x_1,x_2,x_3$ of the sphere (again identified with $\Cc P^1$).
In the coordinate chart where $p\in \Cc_k[z_1,z_2]$ is identified with the one-variable polynomial $z_1 \mapsto p(z_1,1)$, they are given by
\begin{align*}
    &\mathrm{BT}_k(x_1) = \frac{1}{k+2}\left( (1-z_1^2)\frac{d}{dz_1} + k z_1 \right), \\
    &\mathrm{BT}_k(x_2) = \frac{i}{k+2}\left( (1+z_1^2)\frac{d}{dz_1} - k z_1 \right), \\
    &\mathrm{BT}_k(x_3) = \frac{1}{k+2}\left( 2z_1\frac{d}{dz_1} - k \right)\, .
\end{align*}
These operators provide a scaled representation of $\mathfrak{so}(3)$, namely
\begin{equation}\label{eq:BT_repr}
    [\mathrm{BT}_k(x_1),\mathrm{BT}_k(x_2)] = \frac{2i}{k+2}\mathrm{BT}_k(x_3),
\end{equation}
with cyclic permutations.
This reflects the well-known irreducible representation of $\mathfrak{so}(3)$ on $\Cc_k[z_1,z_2]$, corresponding to the $\mathrm{SU}(2)$ group representation where the action of $A\in\mathrm{SU}(2)$ is $(A\cdot p)(z) = p(A^{-1} z)$.
In his original publication, Berezin~\cite{Be1975} proves that any quantization on the sphere that is equivariant under the $\mathrm{SO}(3)$ action is unique up to scaling (it is essentially a consequence of the uniqueness of irreducible representations).
Thus, we get the connection to the quantization matrix operators $T_N(f)$ constructed in section~\ref{sec:quant-via-repr} above by identifying $\Cc_{k}[z_1,z_2] \simeq \Cc^{k+1}$.
First, the relation between $N$ and $k$ must be $N=k+1$.
Second, it is a matter of getting the scaling right.
Recall from section~\ref{sec:quant-via-repr} that $\left(\sum_\alpha T_N(x_\alpha)^2\right)^{1/2} = i I$.
On the other hand, the operators $\mathrm{BT}_k(x_\alpha)$ fulfill
\begin{equation*}
    \sum_{\alpha=1}^3 \mathrm{BT}_k(x_\alpha)^2 = \frac{k}{k+2}.
\end{equation*} 
Thus, we arrive at the following.

\begin{lemma}\label{lem:BT_equivalence}
    For $f\in C^\infty(\Ss^2)$, let $T_N(f)$ denote the matrix quantization operator constructed in section~\ref{sec:quant-via-repr} above. 
    Then, under the identification $\mathbb{C}_{N-1}[z_1,z_2]\simeq \mathbb{C}^N$, 
    \begin{equation*}
        T_N(f) = i\sqrt{\frac{N+1}{N-1}} \mathrm{BT}_{N-1}(f).
    \end{equation*}
\end{lemma}
From \eqref{eq:BT_repr}, the scaling in Lemma~\ref{lem:BT_equivalence} recovers \new{that $\{\cdot,\cdot\}$ is approximated, via $T_N$, by $\frac{1}{\hbar}[\cdot,\cdot]$} with $\hbar$ as in section~\ref{sec:quant-via-repr}.
In particular, since $\sqrt{\frac{N+1}{N-1}} = \mathcal{O}(1)$ as $N\to\infty$, Theorem~\ref{thm:quantization_limit} above by Charles and Polterovich~\cite{ChPo2018} is valid for $T_N$, although it was formulated for $\mathrm{BT}_k$.



\bibliographystyle{plain}
\bibliography{bib-zeitlin}

\end{document}